\documentclass[letterpaper,12pt]{article}

\usepackage[english,activeacute]{babel}
\usepackage{amsmath,amsthm}
\usepackage{amscd}
\usepackage{mathtools}
\usepackage{amssymb,latexsym}
\usepackage{enumerate, appendix}
\usepackage{url}
\usepackage[all]{xy}
\usepackage[mathscr]{euscript}
\usepackage{hyperref}
\usepackage{dcolumn}
\usepackage{amssymb}
\usepackage[utf8]{inputenc}
\usepackage{graphicx}
\usepackage{color}
\usepackage[all]{xy}
\usepackage{tikz}
\usepackage{authblk}
\usepackage{blindtext}
\usepackage{lmodern}
\usepackage{anyfontsize}

\makeatletter
\newcommand{\subjclass}[2][2010]{%
  \let\@oldtitle\@title%
  \gdef\@title{\@oldtitle\footnotetext{#1 \emph{Mathematics subject classification:} #2}}%
}
\newcommand{\keywords}[1]{%
  \let\@@oldtitle\@title%
  \gdef\@title{\@@oldtitle\footnotetext{\emph{Key words and phrases:} #1.}}%
}
\newcommand{\emails}[1]{%
  \let\@@@oldtitle\@title%
  \gdef\@title{\@@@oldtitle\footnotetext{\emph{E-mail addresses:} #1.}}%
}
\makeatother

\newcommand*\circled[1]{\tikz[baseline=(char.base)]{
            \node[shape=circle,draw,inner sep=0.4pt] (char) {#1};}}

\newcommand{\Ss}{{\ensuremath{\mathbb{S}^{1}}}}

\newcommand{\PP}{{\ensuremath{\mathbb{P}}}}
\newcommand{\BB}{{\ensuremath{\mathbb{B}}}}

\newcommand{\F}{\ensuremath{\mathbb{F}}}
\newcommand{\HH}{\ensuremath{\mathbb{H}}}

\newcommand{\Z}{\ensuremath{\mathbb{Z}}}
\newcommand{\Fp}{\ensuremath{\mathbb{F}_{p}}}
\newcommand{\FF}{\ensuremath{\mathbb{F}_{2}}}

\newcounter{numero}

\newcounter{letra}

\newcounter{romnumero}

\newcounter{bibnumero}

\theoremstyle{plain}
   \newtheorem*{teo*}{Theorem}
   \newtheorem*{pro*}{Proposition}
   \newtheorem{pro}{Proposition}[section]
   \newtheorem{teo}[pro]{Theorem}
   \newtheorem{lem}[pro]{Lemma}
   \newtheorem{cor}[pro]{Corollary}
 \theoremstyle{definition}
    \newtheorem{defi}{Definition}[section]
    \newtheorem{exa}{Example}[section]
\theoremstyle{remark}
\newtheorem{remk}{Remark}

\title{\bf A BV-algebra Structure on Hochschild Cohomology of the Group Ring of Finitely Generated Abelian Groups} 

\author[1]{Andrés Angel\footnote{A. Angel is supported in part by the FAPA funds from Vicerrectoría de Investigaciones de la Universidad de los Andes}}
\author[1]{Diego Duarte\footnote{D. Duarte is supported by Fondo de Investigaciones de la Facultad de Ciencias de la Universidad de los Andes. Convocatoria 2017-I para la financiación de proyectos de investigación categoría estudiantes de doctorado candidatos.}}
\affil[1]{Department of Mathematics, Universidad de los Andes, Bogotá, Colombia\\

Carrera 1 No 18A - 12}
\emails{ja.angel908@uniandes.edu.co (A. Angel), dd.duarte22@uniandes.edu.co (D. Duarte)}
\date{\small April 9, 2017}

\begin{document}

\maketitle

\begin{abstract}
We study a Batalin-Vilkovisky algebra structure on the Hochschild cohomology of the group ring of finitely generated abelian groups. The Batalin-Vilkovisky algebra structure for finite abelian groups comes from the fact that the group ring of finite groups is a symmetric algebra, and the Batalin-Vilkovisky algebra structure for free abelian groups of finite rank comes from the fact that its group ring is a Calabi-Yau algebra.  
\end{abstract}

\section{Introduction}

The Hochschild (co)homology of associative algebras has been extensively studied since its first appearance in 1945 with the paper {\it On The Cohomology Groups of an Associative Algebra} by Gerard  Hochschild \cite{Hoch}. There is a rich algebraic structure on the Hochschild cohomology of an associative algebra. It is a graded algebra given by the cup product. In \cite{Ger}, Gerstenhaber proves that the cup product is commutative, and even more that exist a Lie bracket that endows $HH^{*}(A,A)$ with a structure of Lie algebra. These two structures satisfy some compatibility conditions that are now known to define a \textit{Gerstenhaber algebra}. 

In \cite{Tradler}, Tradler proves that if $A$ is a symmetric algebra up to homotopy then $HH^{*}(A,A)$ is a \textit{Batalin-Vilkovisky algebra}. In \cite{M2}, Menichi  presents another proof for Tradler's result for symmetric differential graded algebras. These structures play an important role due to its connection with string topology as can be found in \cite{Bur}, \cite{cohenjonesyan}, \cite{cohenjones}, \cite{FelMenTh}, \cite{FelTh}, \cite{M3}, \cite{Wes1} and \cite{Wes2}.

Given a symmetric algebra, such as a group ring of a finite group, the Batalin-Vilkovisky structure depends on the duality isomorphism, by using different symmetric forms we get different Batalin-Vilkovisky structures with the same underlying Gesternhaber algebra. The Batalin-Vilkovisky algebra structure on the Hochschild cohomology of cyclic groups of prime order over $\F_p$ was calculated by Yang \cite{Yang} using the isomorphism between the group ring and the truncated polynomial ring. However, the symmetric form used on those calculations do not correspond to the canonical form over group rings.  For cyclic groups using the canonical symmetric form, we get  

\begin{teo*}
Let $R$ be an integral domain with $char(R)\nmid n$ and $A=R[\Z/n\Z]$. Then as a BV-algebra 
\begin{align}
	HH^{*}(A;A)&=R [x,z]/(x^{n}-1,nz) \notag \\
		\Delta(a)&=0 \quad \forall a\in HH^{*}(A;A) \notag
\end{align}	
where $|x|=0$ and $|z|=2$. 
\end{teo*}

\begin{teo*}
Let $R$ be a commutative ring with $char(R)=p>0$ and $A=R[\Z/n\Z]$ with $n=mp$. If $p\neq 2$, or $p=2$ and $m$ is even. Then as a BV-algebra 
\begin{align}
HH^{*}(A;A)&=R [x,y,z]/(x^{n}-1,y^{2}) \notag \\
\Delta (z^{k}y^{r}x^{l}) &= r(l-1)z^{k}x^{l-1} \notag 
\end{align}
If $p=2$ and $m$ is odd. Then as a BV-algebra
\begin{align}
HH^{*}(A;A)&=R [x,y,z]/(x^{n}-1,y^{2}-x^{n-2}z) \notag \\
\Delta (z^{k}y^{r}x^{l}) &= r(l-1)z^{k}x^{l-1} \notag 
\end{align}
where $|x|=0$, $|y|=1$ and $|z|=2$. 
\end{teo*}

The aim of this paper is to present a Batalin-Vilkovisky algebra structure on the Hochschild cohomology of the group ring of finitely generated abelian groups. In order to achieve this goal, we study the behavior of the Batalin-Vilkovisky structure for tensor products. Over fields in \cite{Tensor}, Le and Zhou prove that the Künneth formula for Hochschild cohomology is an isomorphism of Gerstenhaber algebras if at least one of the algebras is finite dimensional, and if the algebras are symmetric is an isomorphism of Batalin-Vilkovisky algebras. In section 3, we extend their result for a general class of rings. As a particular case over the integers, we get the following new result 

\begin{teo*}
Let $A=\Z[\Z/n\Z]$ and $B=\Z[\Z/m\Z]$ with $n=km$. Then, as a BV-algebra	
\begin{align}
	HH^{*}(A\otimes B;A\otimes B)&= \frac{\Z[x,t,a,b,c]}{(x^n-1, t^m-1, na, mb, mc, c^2)}  \notag \\
\Delta(x^{i}t^{j}a^{l}b^{r}c^{s}) &= sx^{i-1}t^{j}a^{l}b^{r}((i-1)b-jka)  \notag 
\end{align}
in all cases except when $m$ is even and $k$ is odd, in which case we get
\begin{align}
	HH^{*}(A\otimes B;A\otimes B)&= \frac{\Z[x,t,a,b,c]}{(x^n-1, t^m-1, na, mb, mc, c^2-\frac{m}{2}x^{n-2}ab(b+ka))}  \notag \\
\Delta(x^{i}t^{j}a^{l}b^{r}c^{s}) &= sx^{i-1}t^{j}a^{l}b^{r}((i-1)b-jka)  \notag 
\end{align}
where $|x|=|t|=0$, $|a|=|b|=2$ and $|c|=3$.
\end{teo*}
Notice that the tensor product of the corresponding Hochschild cohomology rings gives a trivial BV-structure. Nevertheless, the Hochschild cohomology of the tensor product gives a highly non-trivial BV-structure.
 
When the algebra is not symmetric but satisfies some sort of Poincaré duality. Ginzburg \cite{Ginz} and Menichi \cite{M2} prove that $HH^*(A;A)$ is also a Batalin-Vilkovisky algebra by transferring the Connes $B$-operator through the isomorphism between Hochschild homology and Hochshild cohomology.  For the tensor product of two such algebras, we prove that if the algebras satisfy some finiteness condition on their resolutions (\ref{isoResAB}), there is also an isomorphism of Batalin-Vilkovisky algebras between the Hochschild cohomology of the tensor product and the tensor product of their cohomologies. In particular, for free abelian groups of finite rank, we have 

\begin{teo*}
As BV-algebras,
\begin{align}
	HH^{*}(R[\Z^{n}]&;R[\Z^{n}])= R[x_1,x_1 ^{-1},\dots,x_n,x_n ^{-1}]\otimes \Lambda(y_1,\dots ,y_n) \notag \\
	\Delta(x_1 ^{i_1}\cdots x_n ^{i_n}y_1 ^{r_1}\cdots y_n ^{r_n}) &= \displaystyle \sum_{k=1} ^{n} (-1)^{^{r_1+\cdots +r_{k-1}}} r_k( i_k-1) x_1 ^{i_1}\cdots x_k^{i_k-1}\cdots x_n ^{i_n}y_1 ^{r_1}\cdots \widehat{y_k ^{r_k}}\cdots y_n ^{r_n}\notag
\end{align}
where $|x_i|=|x_i ^{-1}|=0$ and $|y_i|=1$ for $1\leq i\leq n$.
\end{teo*}


\section{Hochschild (Co)homology}\label{Hoch(co)}

Let $A$ be a $R$-projective $R$-algebra with unit and $R$ be a commutative ring. Denote by $A^{op}$ the opposite algebra of $A$ and by $A^{e}$ the enveloping algebra $A\otimes A^{op}$. Recall that any left and right $A$-module can be considered as a left, or right, $A^{e}$-module. Let $M$ be an $A^{e}$-module.  The \textit{Hochschild homology of $A$ with coefficients in $M$} is 
$$
      HH_{*}(A;M) := Tor_{*} ^{A^{e}}(A;M) 
$$
and the \textit{Hochschild cohomology of $A$ with coefficients in $M$} is
$$      
      HH^{*}(A;M) := Ext^{*} _{A^{e}}(A;M) 
$$

Besides the additive structure, the Hochschild cohomology $HH^{*}(A;A)$ has a graded algebra structure induced from the cup product defined over cochains by 
\begin{equation}\label{usucup}
(f \smile g)(a_1, \dots ,a_{k+j}) =  f (a_1 , \dots , a_{k})g(a_{k+1}, \dots , a_{k+j})
\end{equation}
where $f\in Hom(\bar{A}^k, A)$ and $g\in Hom(\bar{A}^j, A)$.

Since Hochschild cohomology can be computed by using different resolutions. A more general notion of the cup product can be defined as follows. Let $\PP(A)\xrightarrow{\mu}A$ be an $A^{e}$-projective resolution of $A$, and let $\Delta: \PP(A) \rightarrow \PP(A) \underset{A}\otimes \PP(A)$ be a diagonal approximation map, i.e., an $A^{e}$-chain map such that $(\mu\otimes \mu)\circ \Delta=\mu$. If $M$ and $N$ are $A^e$-modules the Hochschild cup product is defined by
\begin{align}
	\smile: HH^{*}(A;M)\otimes HH^{*}(A;N) &\longrightarrow HH^{*}(A;M \underset{A}\otimes N) \notag \\
	\alpha\otimes \beta &\longmapsto (-1)^{|\alpha||\beta|}(\alpha \underset{A}\otimes \beta)\Delta
\end{align}
Notice that if $M=A$ the cup product endows $HH^{*}(A;N)$ with the structure of $HH^{*}(A;A)$-module 
$$
\xymatrix{
	HH^{*}(A;A)\otimes HH^{*}(A;N) \ar[r]^(0.58){\smile} & HH^{*}(A;A \underset{A}\otimes N) \ar[r]^(0.56){\cong} & HH^{*}(A;N)
}
$$   
and if $M=A=N$ the cup product is a product in $HH^{*}(A;A)$
$$
\xymatrix{
	HH^{*}(A;A)\otimes HH^{*}(A;A) \ar[r]^(0.58){\smile} & HH^{*}(A;A \underset{A}\otimes A) \ar[r]^(0.56){\cong} & HH^{*}(A;A)
}
$$  
that will coincide with the one defined over the bar resolution.

\begin{remk}
The diagonal approximation map that recovers the cup product defined on the bar resolution is given by
\begin{align}
 \Delta_{\BB(A)}: \BB(A) &\longrightarrow \BB(A) \underset{A}\otimes \BB(A) \notag \\
 a_{0}\otimes \cdots \otimes a_{n+1} &\longmapsto \sum_{i=0}^{n} a_{0} \otimes \cdots \otimes a_{i}\otimes 1 \underset{A}\otimes 1 \otimes a_{i+1} \otimes \cdots \otimes a_{n+1} 
\end{align}
\end{remk}

\begin{lem}
Let $A$ be a $R$-projective $R$-algebra. Then any Hochschild diagonal approximation map calculates the cup product in $HH^{*}(A;A)$.
\end{lem}

\begin{proof}
Let $\PP(A)\xrightarrow{\mu}A$ be an $A^{e}$-projective resolution of $A$, and let $\Delta: \PP(A) \rightarrow \PP(A) \underset{A}\otimes \PP(A)$ be a diagonal approximation map. We only need to prove that $\PP(A) \underset{A}\otimes \PP(A) \xrightarrow{\mu\otimes \mu}A$ is an $A^{e}$-projective resolution. Since 
\[
	\left(\PP(A) \underset{A}\otimes \PP(A) \right)_{n}=\bigoplus_{i+j=n} P_i\underset{A}\otimes P_j 
\]
and each $P_i$ is $A^{e}$-projective $(P_i\oplus Q\cong \oplus A^{e})$, it suffices to show that $A^{e}\underset{A}\otimes A^{e}$ is $A^{e}$-projective. By hypothesis, $A$ is $R$-projective and $A^{e}\underset{A}\otimes A^{e}\cong A^{e}\otimes A$ as $A^{e}$-modules then $A^{e}\underset{A}\otimes A^{e}$ is $A^{e}$-projective.

Now, to see that the complex is acyclic, notice that each $P_i$ is $A$-projective because $A$ is $R$-projective, and $H_{*}(\PP(A))\cong A$ which is $A$-free then  
\[
	Tor^{A}_{p}(H_s(\PP(A));H_t(\PP(A)))=0 \qquad \forall p\geq 1
\]
and 
\[
	Tor^{A}_{0}(H_s(\PP(A));H_t(\PP(A)))=H_s(\PP)\underset{A}\otimes H_t(\PP)=\begin{cases}
									A\underset{A}\otimes A &\text{if $s=t=0$} \notag\\
									0 &\text{otherwise}
								\end{cases}
\]
Applying the K\"{u}nneth spectral sequence, we get 
$$
H_{*}(\PP\otimes_{A}\PP)\cong A\underset{A}\otimes A\cong A
$$
Since $\PP(A)\xrightarrow{\mu}A$ and $\PP(A) \underset{A}\otimes \PP(A)\xrightarrow{\mu\otimes \mu}A$ are both $A^{e}$-projective resolutions of $A$, by the comparison theorem, $\Delta: \PP(A) \rightarrow \PP(A) \underset{A}\otimes \PP(A)$ exists and it is unique up to homotopy. Therefore, the usual cup product given by the bar resolution (\ref{usucup}) coincides with any other cup product given by different resolutions and diagonal approximation maps.
\end{proof}

Recall that $HH^*(A;A)$ acts on $HH_*(A;A)$. For $n\geq m$, $f\in Hom(\bar{A}^m, A)$ and $a_1\otimes \cdots \otimes a_{n}\otimes a\in \bar{A}^n\otimes A$ the action is given by
$$
	(a_1\otimes \cdots \otimes a_{n}\otimes a)\cdot f= (-1)^{nm} a_{m+1}\otimes \cdots \otimes a_n \otimes a f(a_1\otimes \cdots \otimes a_{m}) 
$$
	
This action can be calculated over any resolution as follows

\begin{pro}\label{actionHH}
Let $A$ be a $R$-projective $R$-algebra and $\Delta$ be any diagonal approximation map. The action of Hochschild cohomology on Hochschild homology is given by
\begin{align}
\rho: HH_{n}(A;A)\otimes HH^m(A;A)&\longrightarrow HH_{n-m}(A;A) \notag \\
		(x \underset{A^e}\otimes a)\otimes f&\longmapsto (-1)^{nm} (f\underset{A}\otimes id)\Delta(x)\underset{A^e}\otimes a \notag 
\end{align}
\end{pro}
 
\begin{proof}
Notice that $f$ is a cochain iff the map $f:\PP(A)\rightarrow A$ is a chain map. Then $\rho$ is well defined because $(f\underset{A}\otimes id)\Delta:\PP(A)\rightarrow \PP(A)$ is a chain map. Since any approximation map is unique up to homotopy, it is sufficient to prove that the formula coincided with the one given for the bar resolution. Let $1\otimes a_1\otimes \cdots \otimes a_n\otimes 1\underset{A^e}\otimes a\in B_{n}(A)$ and $f\in Hom_{A^e}(B_m(A),A)$
\begin{align}
& (-1)^{nm} (f\underset{A}\otimes id)\Delta(1\otimes a_1\otimes \cdots \otimes a_n\otimes 1)\underset{A^e}\otimes a \notag \\
&=(-1)^{nm} (f\underset{A}\otimes id) \left( \sum_{i=0}^{n} 1 \otimes a_1\otimes \cdots \otimes a_{i}\otimes 1 \underset{A}\otimes 1 \otimes a_{i+1} \otimes \cdots \otimes a_{n} \otimes 1 \right) \underset{A^e}\otimes a\notag \\
&=(-1)^{nm} f(1 \otimes a_1\otimes \cdots \otimes a_{m}\otimes 1 ) \underset{A}\otimes 1 \otimes a_{m+1} \otimes \cdots \otimes a_{n} \otimes 1 \underset{A^e}\otimes a \notag 
\end{align}
\end{proof} 
 
In \cite{Ger}, Gerstenhaber proves that the cup product on Hochschild cohomology is graded commutative and that there exists a Lie bracket that endows  $HH^{*}(A;A)$ with a structure of Lie algebra. The Gerstenhaber bracket on $HH^{*}(A;A)$ using the bar resolution is defined as follows
\[
	\lbrace f, g \rbrace = f \circ g - (-1)^{(|f|-1)(|g|-1)} g \circ f
\]
where $\circ$ is defined by
\begin{align}\label{Bracket}
(f \circ g)(a_1\otimes \cdots \otimes a_{k+j-1}) &= \notag \\
\sum_{i=1}^{k}(-1)^{(j-1)(i-1)} f (a_1\otimes \cdots &\otimes a_{i-1}\otimes g(a_i\otimes \cdots \otimes a_{i+j-1})\otimes a_{i+j} \otimes \dots \otimes a_{k+j-1}) \notag
\end{align}

The cup product and the bracket satisfy the following compatibility conditions.
 
\begin{defi}\label{DefGersalg}
A \textit{Gerstenhaber algebra} is a graded commutative algebra $A$ with a linear map $\left\{-,-\right\}: A_{i} \otimes A_{j}\rightarrow A_{i+j-1}$ of degree $-1$ such that
\begin{enumerate}
		\item The bracket $\left\{-,-\right\}$ endows $A$ with a structure of graded Lie algebra of degree $1$, i.e., for all $a,b$ and $c\in A$
					\begin{align}
								&\left\{a,b\right\}=-(-1)^{(\left|a\right|+1)(\left|b\right|+1)}\left\{b,a\right\} \notag \\
							&\left\{a,\left\{b,c\right\}\right\}=\left\{\left\{a,b\right\},c\right\}+(-1)^{(\left|a\right|+1)(\left|b\right|+1)}\left\{b,\left\{a,c\right\}\right\} \notag
					\end{align}
		\item The product and the Lie bracket satisfy the Poisson identity, i.e., for all $a,b$ and $c\in A$
		\[
					\left\{a,bc\right\}=\left\{a,b\right\}c+(-1)^{(\left|a\right|+1)\left|b\right|}b\left\{a,c\right\}
		\]
\end{enumerate}
\end{defi}

If there is a differential of degree $-1$ of a Gerstenhaber algebra such that the Gerstenhaber bracket is the obstruction of the operator to be a graded derivation, then the Gerstenhaber algebra is called a Batalin-Vilkovisky algebra. 

\begin{defi}\label{DefBValg}
A \textit{Batalin-Vilkovisky algebra} is a Gerstenhaber algebra $A$ with a linear map of degree $-1$, $\Delta: A_{i} \rightarrow A_{i-1}$ such that $\Delta \circ \Delta=0$ and 
\[
		\left\{a,b\right\}=-(-1)^{\left|a\right|}(\Delta(ab)-\Delta(a)b-(-1)^{\left|a\right|}a\Delta(b))
\]
for all $a$ and $b\in A$.
\end{defi}

The way to construct BV-structures on Hochschild cohomology is by dualizing or transferring the Connes $B$-operator. 

\begin{defi}\label{BConnesHH}
Let $A$ be a unital algebra. The {\it Connes $B$-operator} is a map on Hochschild homology defined on normalized chains as follows 
\begin{align}
&B_n: \bar{A}^n \otimes A \longrightarrow \bar{A}^{n+1} \otimes A \notag \\
&B_n(a_1\otimes \cdots \otimes a_{n}\otimes a)= \sum^{n}_{i=0} (-1)^{in} a_i\otimes \cdots \otimes a_{n}\otimes a\otimes a_1 \otimes \cdots \otimes a_{i-1}\otimes 1
\end{align} 
\end{defi}

The dual of this operator 
$$
B^{\vee} : Hom(\bar{A}^{*+1}\otimes A, R)\rightarrow Hom(\bar{A}^{*}\otimes A,R)
$$ 
defines by adjunction an operator on $Hom(\bar{A}^{*},A^{\vee})\cong Hom(\bar{A}^{*}\otimes A,R)$, where $A^{\vee}=Hom(A,R)$ When $A$ is a symmetric algebra the non-degenerate bilinear form of $A$ induces a chain complex isomorphism
\[
	Hom(\bar{A}^{*},A^{\vee})\cong Hom(\bar{A}^{*},A)
\] 
which defines a BV-operator, $\Delta$, on the Hochschild cochains.  

\begin{defi}
Let $A$ be a finitely generated projective $R$-algebra. $A$ is called a {\it Frobenius algebra} if there exists an isomorphism of left, or right, $A$-modules  
$$
		\varphi: A \xrightarrow{\cong} A^{\vee}=Hom_R (A,R)
$$ 
If the isomorphism is of $A^e$-modules, $A$ is called a {\it symmetric algebra}. 
\end{defi}

\begin{remk}
Given a Frobenius algebra $A$, it can be defined a non-degenerate bilinear form, 
$$
\langle \cdot , \cdot \rangle : A\otimes A \longrightarrow R
$$
as follows
$$ 
\langle a, b\rangle :=
	 \begin{cases}
    \varphi(b)(a)=\varphi(1)(ab) &\text{if $\varphi$ is a left isomorphism} \notag \\
    \varphi(a)(b)=\varphi(1)(ab) &\text{if $\varphi$ is a right isomorphism}   		\end{cases}			 
$$
Notice that the pairing is associative 
\begin{align}
\langle ab,c\rangle=\varphi(c)(ab)=\varphi(1)(abc)=\varphi(bc)(a)=\langle a,bc\rangle \quad &\text{($\varphi$ left isomorphism)}\notag \\
\langle ab,c\rangle=\varphi(ab)(c)=\varphi(1)(abc)=\varphi(a)(bc)=\langle a,bc\rangle \quad &\text{($\varphi$ right isomorphism)}\notag 
\end{align}
Moreover, if $\varphi$ is a two sided isomorphism the pairing is symmetric
$$
 \langle a,b\rangle=\varphi(a)(b)=\varphi(1)(ab)=\varphi(1)(ba)=\varphi(b)(a)=\langle b,a\rangle
$$
\end{remk}

From now on, an associative nonsingular bilinear form will be called a \textit{Frobenius form}. As in the case over fields, Frobenius algebras over commutative rings can be characterized by Frobenius forms.

\begin{pro}
A finitely generated projective $R$-algebra $A$ is Frobenius if and only if there exists a non-degenerate bilinear form, and it is symmetric if and only if there exists such a form which is also symmetric.
\end{pro}

\begin{exa}\label{frob}
Let $R$ be a commutative ring. If  $G$ is a finite group then the group ring $R\left[G\right]$ is a symmetric algebra with Frobenius form given by  
$$
		\langle \cdot,\cdot \rangle: R\left[G\right]\times R\left[G\right] \longrightarrow R,  \qquad \qquad 	\langle g, h \rangle  =  \begin{cases}
									1 &\text{if $g=h^{-1}$} \notag\\
									0 &\text{otherwise}
								\end{cases}
$$
Notice that the Frobenius form of the group ring $R\left[G\right]$ could be defined by using the canonical augmentation of the group ring, 
$$
	\langle a, b \rangle:=\varepsilon(ab)
$$
where 
$$
	\varepsilon: R\left[G\right] \longrightarrow R, \quad \quad	\varepsilon \left( \sum_{g\in G}  \alpha_{_{g}} g \right)=\alpha_{e}
$$
\end{exa}

In the case when $A$ is a symmetric algebra, the BV-operator, $\Delta$, is defined as follows

\begin{pro}
The operator $\Delta: Hom(\bar{A}^{m+1},A) \rightarrow Hom(\bar{A}^{m},A)$ is given by
$$
	\Delta(f)(a_1,\dots ,a_m)= \sum_{j=1} ^{N} \sum_{i=0} ^{m} (-1)^{im} \langle 1, f(a_i,\dots ,a_n,a^j,a_1,\dots ,a_{i-1})\rangle {a^{j}}^{\vee} 
$$
where $\lbrace a^1,\dots ,a^N\rbrace$ is a basis of $A$ and $\lbrace {a^1}^{\vee},\dots ,{a^N}^{\vee}\rbrace$ is the dual basis with respect to the Frobenius form.
\end{pro}

In \cite{Tradler}, Tradler proves that $\Delta$ induces a BV-structure on $HH^{*}(A;A)$, which furthermore induces the Gerstenhaber structure of $HH^{*}(A;A)$.

\begin{teo}[\cite{Tradler}, \cite{M3}]\label{HHBV}
Let $A$ be a symmetric $R$-algebra. Then $HH^*(A,A)$ is a BV-algebra with $\Delta$ given by the dual of the Connes operator. 
\end{teo}

When the algebra is not symmetric but satisfies some sort of Poincaré duality. It is posible to obtain a BV-algebra structure on Hochschild cohomology by transferring the Connes operator.  

\begin{teo}\label{HHBVact}[\cite{Ginz}, \cite{M2}]
Let $a\in HH_{n}(A,A)$ such that 
\begin{align}
\rho_a:HH^*(A;A)&\longrightarrow HH_{n-*}(A;A) \notag \\
		b&\longmapsto \rho(a\otimes b) \notag 
\end{align}
is an isomorphism. If $B(a)=0$ then $HH^*(A,A)$ is a BV-algebra with $\Delta$ given by $\Delta_a:=\rho_{a}^{-1}B\rho_{a}$.
\end{teo}

%
%
%

\section{Hochschild (Co)homology for Tensor Products}\label{TensorHH(co)}

In \cite{Tensor}, Le and Zhou prove the following 

\begin{teo}[\cite{Tensor} Theorem 3.3]\label{AxB}
Let $R$ be a field and $A$ and $B$ be two $R$-algebras such that one of them is finite dimensional. Then there is an isomorphism of Gerstenhaber algebras
$$
	HH^{*}(A\otimes B; A\otimes B) \cong HH^{*}(A; A)\otimes HH^{*}(B; B) 
$$
If furthermore, $A$ and $B$ are finite dimensional symmetric algebras, the above isomorphism becomes an isomorphism of Batalin-Vilkovisky algebras.
\end{teo}

In this section, we extend their result for a general class of rings and present an analogous for algebras that satisfy some sort of Poincaré duality.

\begin{pro}
Let $A$ and $B$ be $R$-projective $R$-algebras with $R$ a commutative ring. Suppose that $\PP(A)\rightarrow A$ is an $A^{e}$-projective resolution of $A$ and $\PP(B)\rightarrow B$ is a $B^{e}$-projective resolution of $B$. Then 
\[
	\PP(A\otimes B):=\PP(A)\otimes \PP(B) \longrightarrow A\otimes B
\]
is an $(A\otimes B)^{e}$-projective resolution of $A\otimes B$.
\end{pro}

\begin{proof}
Since 
\[
	\PP_n(A\otimes B)=\bigoplus_{i+j=n} P_i(A)\otimes P_j(B) 
\]
and $A^{e}\otimes B^{e}\cong (A\otimes B)^{e}$, $\PP(A)\otimes \PP(B)\rightarrow A\otimes B$ is an $(A\otimes B)^{e}$-projective complex of $A\otimes B$. It only remains to check that the complex is acyclic. Since $H_{*}(\PP(A))\cong A$ and $H_{*}(\PP(B))\cong B$ which are $R$-projective. Then  
\[
	Tor^{R}_{p}(H_s(\PP(A));H_t(\PP(B)))=0 \qquad \forall p\geq 1
\]
and 
\[
	Tor^{R}_{0}(H_s(\PP(A));H_t(\PP(B)))=H_s(\PP(A))\otimes H_t(\PP(B))=\begin{cases}
									A\otimes B &\text{if $s=t=0$} \notag\\
									0 &\text{otherwise}
								\end{cases}
\]
Applying the K\"{u}nneth spectral sequence, we have 
$$
H_{*}(\PP(A) \otimes \PP(B))\cong A\otimes B
$$
Therefore, $\PP(A)\otimes \PP(B)\rightarrow A\otimes B$ is an $(A\otimes B)^{e}$-projective resolution of $A\otimes B$.
\end{proof}

\begin{pro}
The following map is an isomorphism of complexes
\begin{align}
\tau: (\PP(A)\underset{A}\otimes \PP(A))\otimes (\PP(B)\underset{B}\otimes \PP(B)) &\longrightarrow \PP(A\otimes B) \underset{A\otimes B}\otimes  \PP(A\otimes B) \notag \\
a_1\underset{A}\otimes a_2 \otimes b_1\underset{B}\otimes b_2 &\longmapsto (-1)^{|a_2||b_1|}a_1\otimes b_1 \underset{A\otimes B}\otimes a_2\otimes b_2\notag 
\end{align}
\end{pro}

\begin{proof}
Let $a_1, a_2\in \PP(A)$ and $b_1,b_2\in \PP(B)$ with $|a_1|=i$, $|a_2|=j$, $|b_1|=k$ and $|b_2|=l$.
\begin{align}
\tau \delta_n(a_1 \underset{A}\otimes a_2 \otimes b_1\underset{B}\otimes &b_2) = \tau ( \partial^{A}_{i+j}(a_1 \underset{A}\otimes a_2)\otimes b_1\underset{B}\otimes b_2 \notag \\
&\quad + (-1)^{i+j}a_1\underset{A}\otimes a_2 \otimes \partial^{B}_{k+l}(b_1\underset{B}\otimes b_2) ) \notag \\
&= \tau((d^A_i(a_1)\underset{A}\otimes a_2 + (-1)^i a_1\underset{A}\otimes d^A_j(a_2))\otimes b_1\underset{B}\otimes b_2 \notag \\
&\quad + (-1)^{i+j}a_1\underset{A}\otimes a_2 \otimes (d^B_k(b_1)\underset{B}\otimes b_2 + (-1)^k b_1\underset{B}\otimes d^B_l(b_2))) \notag \\
&= (-1)^{kj} d^A_i(a_1)\otimes b_1 \underset{A\otimes B}\otimes a_2\otimes b_2 \notag \\
&\quad + (-1)^{i+ k(j-1)} a_1\otimes b_1 \underset{A\otimes B}\otimes d^A_j(a_2)\otimes b_2 \notag \\
&\quad + (-1)^{i+ kj} a_1\otimes d^B_k(b_1) \underset{A\otimes B}\otimes a_2\otimes b_2 \notag \\
&\quad + (-1)^{i+j+k(j+1)} a_1\otimes b_1 \underset{A\otimes B}\otimes a_2\otimes d^B_l(b_2) \notag \\
&= (-1)^{kj}((d^A_i(a_1)\otimes b_1 + (-1)^{i} a_1\otimes d^B_k(b_1))\underset{A\otimes B}\otimes a_2\otimes b_2) \notag \\
&\quad + (-1)^{kj+i+k} a_1\otimes b_1 \underset{A\otimes B}\otimes (d^A_j(a_2)\otimes b_2 + (-1)^{j} a_2\otimes d^B_l(b_2)) \notag \\
&= (-1)^{kj}(d^{\otimes}_{i+k}(a_1\otimes b_1)\underset{A\otimes B}\otimes a_2\otimes b_2 \notag \\
&\quad + (-1)^{i+k}a_1\otimes b_1\underset{A\otimes B}\otimes d^{\otimes}_{j+l}(a_2\otimes b_2))\notag \\
&= (-1)^{kj}\partial^{\otimes}_n(a_1\otimes b_1 \underset{A\otimes B}\otimes a_2\otimes b_2) \notag \\
&= \partial^{\otimes}_n \tau(a_1\underset{A}\otimes a_2 \otimes b_1\underset{B}\otimes b_2) \notag
\end{align}
Therefore, $\tau$ is a map of complexes and it is clear that is an isomorphism in each degree, since the inverse of $\tau$ is $\tau$ itself. 
\end{proof}

\begin{pro}\label{TenDiag}
Let $\Delta^{A}:\PP(A)\rightarrow \PP(A)\underset{A}\otimes \PP(A)$ and $\Delta^{B}:\PP(B)\rightarrow \PP(B)\underset{B}\otimes \PP(B)$ be diagonal approximation maps. Then 
\[
	\Delta: \PP(A\otimes B) \xrightarrow{\Delta^{A}\otimes \Delta^{B}} (\PP(A)\underset{A}\otimes \PP(A))\otimes (\PP(B)\underset{B}\otimes \PP(B)) \xrightarrow{\tau} \PP(A\otimes B)\underset{A\otimes B}\otimes \PP(A\otimes B)  
\]
is a diagonal approximation map for $A\otimes B$.
\end{pro}

\begin{proof}
Let $a\in \PP(A)$ and $b\in \PP(B)$ with $|a|=i$ and $|b|=j$
\begin{align}
\partial^{\otimes}_{i+j}\Delta_{i+j}(a\otimes b)&= \partial^{\otimes}_{i+j}\tau(\Delta^{A}_i(a)\otimes \Delta^{B}_j(b)) = \tau \delta_{i+j}(\Delta^{A}_i(a)\otimes \Delta^{B}_j(b)) \notag \\
&= \tau (\partial^{A}_{i}\Delta^{A}_i(a)\otimes \Delta^{B}_j(b) + (-1)^i \Delta^{A}_i(a)\otimes \partial^{B}_{j}\Delta^{B}_j(b)) \notag \\
&= \tau (\Delta^{A}_{i-1}d^{A}_{i}(a)\otimes \Delta^{B}_j(b) + (-1)^i \Delta^{A}_i(a)\otimes \Delta^{B}_{j-1}d^{B}_{j}(b)) \notag \\
&= \Delta_{i+j-1}(d^{A}_{i}(a)\otimes b + (-1)^i a\otimes \Delta^{B}_{j-1}d^{B}_{j}(b)) \notag \\
&= \Delta_{i+j-1} d^{\otimes}_{i+j}(a\otimes b) \notag
\end{align}
For $|a|=|b|=0$, we have
\[
	((\mu_A\otimes\mu_B)\underset{A\otimes B}\otimes (\mu_A\otimes\mu_B))\tau(\Delta^{A}_0\otimes \Delta^{B}_0)(a\otimes b)= (\mu_A\otimes\mu_B))(a\otimes b)
\]
\end{proof}

\begin{teo}\label{injection}
Let $A$ and $B$ be $R$-projective $R$-algebras with $R$ a commutative hereditary ring. Suppose that $\PP(A)\rightarrow A$ is a resolution of $A$ of finitely generated projective $A^{e}$-modules and $\PP(B)\rightarrow B$ is a $B^{e}$-resolution of $B$ such that  
\begin{equation}\label{isoResAB}
	Hom_{(A\otimes B)^{e}}(\PP(A\otimes B), A\otimes B)\cong Hom_{A^{e}}(\PP(A), A)\otimes Hom_{B^{e}}(\PP(B), B) 
\end{equation}
Then  
\[
	HH^{*}(A; A)\otimes HH^{*}(B;B) \hookrightarrow HH^{*}(A\otimes B; A\otimes B)
\]
is an injection of graded algebras.
\end{teo}

\begin{proof}
By K\"{u}nneth theorem, there is an injective map of modules. Let $\Delta^{A}:\PP(A)\rightarrow \PP(A)\underset{A}\otimes \PP(A)$ and $\Delta^{B}:\PP(B)\rightarrow \PP(B)\underset{B}\otimes \PP(B)$ be diagonal approximation maps. By proposition \ref{TenDiag}, $\Delta=\tau (\Delta^{A}\otimes \Delta^{B})$ is a diagonal approximation map for $\PP(A\otimes B)$. Let $f,f'\in Hom_{A^{e}}(\PP(A), A)$ and $g,g'\in Hom_{B^{e}}(\PP(B), B)$. Notice that the following diagram commutes
\[
\xymatrix{
\PP(A)\otimes \PP(B) \ar[d]_{\Delta^A\otimes \Delta^B} & \\
\PP(A)\underset{A}\otimes \PP(A)\otimes \PP(B)\underset{B}\otimes \PP(B) \ar[d]_{(-1)^{|f'||g|}f\underset{A}\otimes f'\otimes g\underset{B}\otimes g'} \ar[r]^-{\tau} & \left(\PP(A\otimes B)\right)\underset{A\otimes B}\otimes \left(\PP(A \otimes  B)\right) \ar[d]^{f\otimes g\underset{A\otimes B}\otimes f'\otimes g'} \\
 A\underset{A}\otimes A\otimes B\underset{B}\otimes B \ar[r]^{\tau} & A\otimes B \underset{A\otimes B}\otimes A\otimes B
}
\]
Therefore,
\begin{align}
((f\otimes g)\smile (f'\otimes g')) &= (-1)^{(|f|+|g|)(|f'|+|g'|)}((f\otimes g)\underset{A\otimes B}\otimes (f'\otimes g'))\Delta \notag \\
&= (-1)^{|f'||g|+|f||f'|+|g||g'|} ((f \underset{A}\otimes f')\Delta^{A}) \otimes ((g \underset{B} \otimes g')\Delta^{B})\notag \\
&= (-1)^{|f'||g|}(f \smile_A f') \otimes (g \smile_B g') \notag
\end{align} 
\end{proof}

\begin{cor}\label{isoalgebras}
Under the same hypothesis as in theorem \ref{injection}, if $HH^{*}(A; A)$, or $H^{*}(B;B)$, is $R$-projective. Then  
\[
	HH^{*}(A; A)\otimes HH^{*}(B;B) \cong HH^{*}(A\otimes B; A\otimes B)
\]
as graded algebras.
\end{cor}

\begin{proof}
By K\"{u}nneth Theorem, there is an isomorphim of modules 
\begin{align}
HH^{n}(A\otimes B; A\otimes B)\cong &\bigoplus_{r+s=n} HH^{r}(A; A)\otimes HH^{s}(B;B) \notag \\
&\oplus \bigoplus_{r+s=n+1} Tor_1 ^{R} (HH^{r}(A; A), HH^{s}(B;B)) \notag
\end{align}
Since $HH^{*}(A; A)$, or $H^{*}(B;B)$, is $R$-projective, the proof of  Theorem \ref{injection} extends to an isomorphism of graded algebras.
\end{proof}

\begin{defi}
A $(p,q)${\it -shuffle} is a sequence of integers 
$$
	[i_1\cdots i_{p}|j_1\cdots j_q]
$$
represented by a permutation $\sigma\in S_{p+q}$, such that  
$$
	\sigma(1)=i_1<\cdots <i_p=\sigma(p) \quad \text{and} \quad \sigma(p+1)=j_1<\cdots <j_q=\sigma(p+q)
$$ 
The sign of a $(p,q)$-shuffle is defined by
$$
	|\sigma|:=|\lbrace (i,j)|1\leq i<j\leq p+q \text{ and } \sigma(i)>\sigma(j)\rbrace|
$$
The set of $(p,q)$-shuffles will be denoted by $S_{p,q}$.	 
\end{defi}

\begin{defi}
The {\it Alexander-Whitney map} $AW: \BB(A\otimes B) \rightarrow \BB(A)\otimes \BB(B)$ is defined as follows
\begin{align}
AW_0(a_1\otimes b_1\otimes a_2\otimes b_2) &= a_1\otimes a_2\otimes b_1\otimes b_2 \notag \\
AW_r(1\otimes 1 \otimes a_1\otimes b_1\otimes \cdots \otimes a_r\otimes b_r\otimes 1\otimes 1)&= \notag \\
\sum_{t=0}^{r} (-1)^{t(r-t)}a_1a_2\cdots a_t\otimes a_{t+1}\otimes \cdots \otimes a_r &\otimes 1\otimes 1\otimes b_1\otimes \cdots \otimes b_t\otimes b_{t+1}\cdots b_r \notag 
\end{align}
for $r\geq 1$, and by convention for $t=0$, $a_1\cdots a_t =1$ and for $t=r$, $b_{t+1}\cdots b_r =1$. 

The {\it Eilenberg-Zilber map} $EZ: \BB(A)\otimes \BB(B) \rightarrow \BB(A\otimes B)$ is defined as follows
\begin{align}
EZ_0(a_1\otimes a_2\otimes b_1\otimes b_2) &= a_1\otimes b_1\otimes a_2\otimes b_2 \notag \\
EZ_r(1\otimes a_1\otimes \cdots \otimes a_{r-t}\otimes 1\otimes 1 \otimes b_1\otimes \cdots b_t \otimes 1)&= \notag \\
\sum_{\sigma\in S_{r-t,t}} (-1)^{|\sigma|} 1\otimes 1\otimes F(x_{\sigma^{-1}(1)})&\otimes \cdots \otimes F(x_{\sigma^{-1}(r)}) \otimes 1\otimes 1 \notag 
\end{align}
for $r\geq 1$, where $F(a)=a\otimes 1$ and $F(b)=1\otimes b$.
\end{defi}

\begin{remk}\label{AWEZid}
These two maps gives an equivalence of complexes. Moreover, 
$$
AWEZ=id \quad \text{and} \quad EZAW\simeq id
$$
\end{remk}

\begin{pro}
The induced maps for $AW_*$ and $EZ_*$ are
\begin{align}
&\overline{AW}_n: (A\otimes B)^n \otimes A\otimes B  \longrightarrow \displaystyle \bigoplus_{i+j=n} A^i \otimes A \otimes B^j \otimes B \notag \\
&\overline{AW}_0\equiv id \notag \\
&\overline{AW}_n((a_1\otimes b_1\otimes \cdots \otimes a_n \otimes b_n) \otimes a\otimes b)= \notag \\
&\qquad \sum_{k=0}^{n} (-1)^{k(n-k)}(a_{k+1}\otimes \cdots \otimes a_n\otimes a a_1a_2\cdots a_k) \otimes (b_1\otimes \cdots \otimes b_k\otimes b_{k+1}\cdots b_n b) \notag \\
&\overline{EZ}_n: \displaystyle \bigoplus_{i+j=n} A^i\otimes A \otimes B^j\otimes B \longrightarrow (A\otimes B)^n\otimes A\otimes B \notag \\
&\overline{EZ}_0\equiv id \notag \\
&\overline{EZ}_n((a_1\otimes \cdots \otimes a_{n-t}\otimes a) \otimes ( b_1\otimes \cdots \otimes b_t \otimes b))= \notag \\
&\qquad \sum_{\sigma\in S_{n-t,t}} (-1)^{|\sigma|} \left( F(x_{\sigma^{-1}(1)})\otimes \cdots \otimes F(x_{\sigma^{-1}(n)})\right)\otimes a\otimes b \notag 
\end{align}
\end{pro}

The Connes $B$-operator on the Hochschild homology of the tensor product of two algebras satisfies the following equation 

\begin{pro}\label{TensorConnes}
$
\overline{AW}B^{^{A\otimes B}}\overline{EZ}=B^{^{A}}\otimes id + id\otimes B^{^{B}}
$
\end{pro}

\begin{proof}
Let $(a_1\otimes \cdots \otimes a_{n-t}\otimes a) \otimes (b_1\otimes \cdots \otimes b_t\otimes b)\in A^{n-t}\otimes A\otimes B^t\otimes B$. Applying $B^{^{A\otimes B}}_n\overline{EZ}_n$, we get
$$
 \sum_{\sigma\in S_{n-t,t}} \sum^n_{i=0}(-1)^{|\sigma|+in} F_i\otimes \cdots \otimes F_n \otimes a\otimes b \otimes F_1 \otimes \cdots \otimes F_{i-1}\otimes 1\otimes 1
$$
where $F_i=F(x_{\sigma^{-1}(i)})$ for $1\leq i\leq n$. Reordering the inner sum, we get 
\begin{align}
\sum_{\sigma} \Bigg( &(-1)^{|\sigma|}a\otimes b \otimes F_1 \otimes \cdots \otimes F_{n}\otimes 1\otimes 1 \tag*{\circled{1}} \\
&+ \sum^{n-t}_{i=1}(-1)^{|\sigma|+in} F_i\otimes \cdots \otimes F_n \otimes a\otimes b \otimes F_1 \otimes \cdots \otimes F_{i-1}\otimes 1\otimes 1 \tag*{\circled{2}} \\
&+(-1)^{|\sigma|+(n-t+1)n} F_{n-t+1}\otimes \cdots \otimes F_n \otimes a\otimes b \otimes F_1 \otimes \cdots \otimes F_{n-t}\otimes 1\otimes 1 \tag*{\circled{3}}  \\
&+ \sum^t_{i=2}(-1)^{|\sigma|+(n-t+i)n} F_{n-t+i}\otimes \cdots \otimes F_n\otimes   \notag \\ 
& \qquad \qquad\otimes a\otimes b \otimes F_1 \otimes \cdots \otimes F_{n-t+i-1} 1\otimes 1 \Bigg)  \tag*{\circled{4}}
\end{align}
Consider the following permutations 
\begin{align}
\sigma_i &=
\begin{pmatrix}
1 & \cdots & i-1 & i & \cdots & n-t & n-t+1 & \cdots & n \\
1 & \cdots & i-1 & i+t & \cdots & n & i & \cdots & i+t-1
\end{pmatrix} \notag \\
 \tag*{$1\leq i\leq n-t$} \\
\tilde{\sigma}_j &=
\begin{pmatrix}
1 & \cdots & n-t & n-t+1 & \cdots & n-t+j-1 & n-t+j & \cdots & n \\
j & \cdots & n-t+j-1 & 1 & \cdots & j-1 & n-t+j & \cdots & n
\end{pmatrix} \notag \\
\tag*{$2\leq j\leq t$}
\end{align}
Notice that $|\sigma_i|=(n-t-i+1)t$ and $|\tilde{\sigma}_i|=(n-t)(j-1)$. Now, applying $\overline{AW}_{n+1}$ to \circled{1} the only non-zero term arises when $\sigma=\sigma_1$ and $k=0$
$$
\overline{AW}_{n+1}(\circled{1})=(-1)^{n-t}(a_1\otimes \cdots \otimes a_{n-t}\otimes a) \otimes (b\otimes b_1\otimes \cdots \otimes b_t\otimes 1)  
$$
Applying $\overline{AW}_{n+1}$ to \circled{2} the only non-zero terms arise when $\sigma=\sigma_i$ for $1\leq i\leq n-t$ and $k=t$
\begin{align}
\overline{AW}_{n+1}(\circled{2})&=\sum^{n-t}_{i=1} (-1)^{i(n-t)} (a_i\otimes \cdots \otimes a_{n-t}\otimes a\otimes a_1\otimes \cdots \otimes a_{i-1}\otimes 1)\notag \\
&\qquad \qquad \otimes (b_1\otimes \cdots \otimes b_t\otimes b) \notag 
\end{align}
Applying $\overline{AW}_{n+1}$ to \circled{3} the only non-zero terms arise when $\sigma=\sigma_i$ for $1\leq i\leq n-t$, $k=t$ and $k=t+1$
\begin{align}
\overline{AW}_{n+1}(\circled{3})=& (a\otimes a_1\otimes \cdots \otimes a_{n-t}\otimes 1) \otimes (b_1\otimes \cdots \otimes b_t\otimes b) \notag \\
&+(-1)^n(a_1\otimes \cdots \otimes a_{n-t}\otimes a) \otimes (b_1\otimes \cdots \otimes b_t\otimes b\otimes 1) \notag
\end{align}
Applying $\overline{AW}_{n+1}$ to \circled{4} the only non-zero terms arise when $\sigma=\tilde{\sigma}_i$ for $2\leq i\leq t$ and $k=t+1$
\begin{align}
\overline{AW}_{n+1}(\circled{4})&=\sum^{t}_{i=2} (-1)^{it+n-t} (a_1\otimes \cdots \otimes a_{n-t}\otimes a)  \notag \\
&\qquad \qquad \otimes (b_i\otimes \cdots \otimes b_{t}\otimes b\otimes b_1\otimes \cdots \otimes b_{i-1}\otimes 1) \notag 
\end{align}
Applying $B^{^{A}}_{n-t}\otimes id + (-1)^{n-t} id\otimes B^{^{B}}_t$ to $( a_1\otimes \cdots \otimes a_{n-t}\otimes a) \otimes (b_1\otimes \cdots \otimes b_t\otimes b)$, we get
\begin{align}
&\sum^{n-t}_{i=0} (-1)^{i(n-t)} (a_i\otimes \cdots \otimes a_{n-t}\otimes a\otimes a_1\otimes \cdots \otimes a_{i-1}\otimes 1)\otimes (b_1\otimes \cdots \otimes b_t\otimes b) \notag \\
&+(-1)^{n-t}\sum^{t}_{i=0} (-1)^{it} (a_1\otimes \cdots \otimes a_{n-t}\otimes a) \otimes (b_i\otimes \cdots \otimes b_{t}\otimes b\otimes \cdots \otimes b_{i-1}\otimes 1) \notag \\
= &(a\otimes a_1\otimes \cdots \otimes a_{n-t}\otimes 1) \otimes (b_1\otimes \cdots \otimes b_t\otimes b) \notag \\
&+ \sum^{n-t}_{i=1} (-1)^{i(n-t)} (a_i\otimes \cdots \otimes a_{n-t}\otimes a\otimes a_1\otimes \cdots \otimes a_{i-1}\otimes 1)\otimes (b_1\otimes \cdots \otimes b_t\otimes b) \notag \\
&+ (-1)^{n-t}(a_1\otimes \cdots \otimes a_{n-t}\otimes a) \otimes (b\otimes b_1\otimes \cdots \otimes b_t\otimes 1) \notag \\
&+\sum^{t}_{i=2} (-1)^{it+n-t} (a_1\otimes \cdots \otimes a_{n-t}\otimes a) \otimes (b_i\otimes \cdots \otimes b_{t}\otimes b\otimes b_1\otimes \cdots \otimes b_{i-1}\otimes 1) \notag \\
&+(-1)^n(a_1\otimes \cdots \otimes a_{n-t}\otimes a) \otimes (b_1\otimes \cdots \otimes b_t\otimes b\otimes 1) \notag \\
=& \overline{AW}_{n+1}(\circled{1}) +\overline{AW}_{n+1}(\circled{2})+\overline{AW}_{n+1}(\circled{3})+\overline{AW}_{n+1}(\circled{4}) \notag 
\end{align}
Therefore,
\begin{align}
\overline{AW}_{n+1}B^{^{A\otimes B}}_n&\overline{EZ}_n((a_1\otimes \cdots \otimes a_{n-t}\otimes a) \otimes (b_1\otimes \cdots \otimes b_t\otimes b)) \notag \\
&=\left(B^{^{A}}_{n-t}\otimes id + (-1)^{n-t} id\otimes B^{^{B}}_t\right) \left((a_1\otimes \cdots \otimes a_{n-t}\otimes a) \otimes (b_1\otimes \cdots \otimes b_t\otimes b)\right) \notag 
\end{align}
\end{proof}

\begin{teo}\label{BVinj}
Let $A$ and $B$ be finite dimensional symmetric $R$-algebras with $R$ a commutative hereditary ring. Then  
\[
	HH^{*}(A; A)\otimes HH^{*}(B;B) \hookrightarrow HH^{*}(A\otimes B; A\otimes B)
\]
is an injection of BV-algebras. 
\end{teo}

\begin{proof}
Since both algebras are finite dimensional, we have 
$$
Hom_{(A\otimes B)^{e}}(\BB(A\otimes B), A\otimes B)\cong Hom_{A^{e}}(\BB(A), A)\otimes Hom_{B^{e}}(\BB(B), B) 
$$
Therefore, by theorem \ref{injection} there is an injection of graded algebras 
\[
	HH^{*}(A; A)\otimes HH^{*}(B;B) \hookrightarrow HH^{*}(A\otimes B; A\otimes B)
\]
By theorem \ref{HHBV}, the BV-operator is given by the dual of the Connes operator. By dualizing equation \ref{TensorConnes}, we get    
$$
\overline{EZ}^{\vee} \Delta^{A\otimes B} \overline{AW}^{\vee} = \Delta^{A}\otimes id+ id\otimes \Delta^{B}
$$
on the cochain level, which gives the desire injection on the cohomological level.  
\end{proof}

\begin{cor}[\cite{Tensor} Theorem 3.5]\label{tensorBViso}
Let $A$ and $B$ be finite dimensional symmetric $R$-algebras with $R$ a commutative hereditary ring. If $HH^{*}(A; A)$, or $H^{*}(B;B)$, is $R$-projective. Then  
\[
	HH^{*}(A; A)\otimes HH^{*}(B;B) \cong HH^{*}(A\otimes B; A\otimes B)
\]
is an isomorphism of BV-algebras. 
\end{cor}

Next, we study the action of Hochschild cohomology on Hochschild homology of tensor products

\begin{pro}\label{tensoractionfinite}
If at least one of the algebras is finite dimensional, the action of $HH^{*}(A\otimes B;A\otimes B)$ on $HH_{*}(A\otimes B;A\otimes B)$ is given by the tensor product of the actions.
\end{pro}

\begin{proof}
Let $(a_1\otimes \cdots \otimes a_{n-t}\otimes a) \in A^{n-t}\otimes A$, $(b_1\otimes \cdots \otimes b_t\otimes b)\in \otimes B^t\otimes B$, $\alpha\in Hom(A^{m-i},A)$ and $\beta\in Hom(B^{i},B)$ with $n-t\geq m-i \geq 0$ and $n-m\geq t-i \geq 0$. We claim that 
$$
\overline{AW}\rho^{A\otimes B} (\overline{EZ}\otimes \overline{AW}^{\vee})=\pm\left(\rho^{A}\otimes \rho^{B}\right) 
$$
on the (co)chain level, which implies the assertion on the (co)homological level. 
\begin{align}
&(a_1\otimes \cdots \otimes a_{n-t}\otimes a)\otimes (b_1\otimes \cdots \otimes b_t\otimes b)\otimes (\alpha \otimes \beta) \notag \\
&\xmapsto{\overline{EZ}_n\otimes \overline{AW}^{\vee}_m} \left(\sum_{\sigma}(-1)^{t(m-i)+|\sigma|} F_1\otimes \cdots \otimes F_{n}\otimes a\otimes b\right)\otimes \overline{AW}^{\vee}_m(\alpha\otimes \beta) \xmapsto{\rho^{A\otimes B}} \notag \\
&\sum_{\sigma}(-1)^{t(m-i)+|\sigma|+nm} F_{m+1}\otimes \cdots \otimes F_{n} \otimes (a\otimes b) \; \overline{AW}^{\vee}_m(\alpha\otimes \beta)(F_1\otimes \cdots \otimes F_m) \notag
\end{align}
Applying $\overline{AW}_{n-m}$ the only non-zero term arise when $k=t-i$ and $\sigma$ is the following permutation
\begin{equation}
\resizebox{\textwidth}{!}{$
\begin{pmatrix}
1 &\cdots & m-i & m-i+1 & \cdots & n-t & n-t+1 & \cdots & n-t+i & n-t+i+1  \cdots  n\\
i+1 & \cdots & m & m+t-i+1 & \cdots & n & 1 & \cdots & i & m+1  \cdots  m+t-i
\end{pmatrix}$} \notag
\end{equation}
Since $|\sigma|=i(m-i)+t(n-t-m+i)$, we get
\begin{align}
(-1)^{t(m-i)+(m-i)(n-t)+it}(a_{m-i+1}\otimes \cdots \otimes a_{n-t}\otimes a\, \alpha(a_1\otimes \cdots \otimes a_{m-i}))\notag \\
\otimes (b_{i+1}\otimes \cdots \otimes b_{t}\otimes b\, \beta(b_1\otimes \cdots \otimes b_{i})) \notag 
\end{align}
which is precisely 
$$
(-1)^{t(m-i)}\left(\rho^A\otimes \rho^B\right)\left((a_1\otimes \cdots \otimes a_{n-t}\otimes a)\otimes \alpha \otimes (b_1\otimes \cdots \otimes b_t\otimes b) \otimes \beta \right)
$$
\end{proof}

The following proposition is a slightly generalization of the previous proposition \ref{tensoractionfinite}

\begin{pro}\label{tensoraction}
Under the same hypothesis as in theorem \ref{injection}, the action of $HH^{*}(A\otimes B;A\otimes B)$ on $HH_{*}(A\otimes B;A\otimes B)$ is given by the tensor product of the actions.
\end{pro}

\begin{proof}
Let $\Delta^{A}:\PP(A)\rightarrow \PP(A)\underset{A}\otimes \PP(A)$ and $\Delta^{B}:\PP(B)\rightarrow \PP(B)\underset{B}\otimes \PP(B)$ be diagonal approximation maps. By proposition \ref{TenDiag}, $\Delta=\tau (\Delta^{A}\otimes \Delta^{B})$ is a diagonal approximation map for $\PP(A\otimes B)$. Let $x\otimes a\in \PP(A)\underset{A^e}\otimes A$, $y\otimes b\in \PP(B)\underset{B^e}\otimes B$, $f\in Hom_{A^{e}}(\PP(A), A)$ and $g\in Hom_{B^{e}}(\PP(B), B)$. Notice that the following diagram commutes up to the sign $(-1)^{|x||g|}$
$$
\xymatrix{
\PP(A)\otimes \PP(B) \ar[d]_{\Delta^A\otimes \Delta^B} & \\
\PP(A)\underset{A}\otimes \PP(A)\otimes \PP(B)\underset{B}\otimes \PP(B) \ar[d]_{(-1)^{|x||g|} (f\underset{A}\otimes id) \otimes (g\underset{B}\otimes id)} \ar[r]^-{\tau} & \PP(A\otimes B)\underset{A\otimes B}\otimes \PP(A \otimes B) \ar[d]^{(f\otimes g)\underset{A\otimes B}\otimes id} \\
\PP(A)\underset{A}\otimes A \otimes \PP(B)\underset{B}\otimes B \ar[r]^{\tau} & \PP(A\otimes B) \underset{A\otimes B}\otimes A\otimes B 
}
$$
Therefore,
$$
\rho^{A\otimes B}(((x\otimes a)\otimes (y\otimes b))\otimes (f\otimes g))=(-1)^{|f||y|}\rho^{A}((x\otimes a)\otimes f)\otimes \rho^{B}((y\otimes b)\otimes g)
$$
\end{proof}

To sum up, we get the following

\begin{teo}\label{BVisoCY}
Let $R$ be a commutative hereditary ring. Let $A$ and $B$ be two $R$-algebras satisfying the following hypothesis: 
\begin{itemize}
\item Suppose that $\PP(A)\rightarrow A$ is a resolution of $A$ of finitely generated projective $A^{e}$-modules and $\PP(B)\rightarrow B$ is a $B^{e}$-resolution of $B$ such that  
$$
	Hom_{(A\otimes B)^{e}}(\PP(A\otimes B), A\otimes B)\cong Hom_{A^{e}}(\PP(A), A)\otimes Hom_{B^{e}}(\PP(B), B) 
$$
\item $HH^{*}(A; A)$, or $H^{*}(B;B)$, is $R$-projective.
\item Let $a\in HH_n(A;A)$ and $b\in HH_m(B;B)$ such that
\begin{align}
\rho_a:HH^*(A;A)&\longrightarrow HH_{n-*}(A;A) & \rho_b:HH^*(B;B)&\longrightarrow HH_{m-*}(B;B) \notag \\
		c&\longmapsto \rho(a\otimes c) & c&\longmapsto \rho(b\otimes c)\notag 
\end{align}
are isomorphisms, and $B(a)=0=B(b)$.
\end{itemize}
Then there is an isomorphism of BV-algebras
\[
	HH^{*}(A\otimes B) \cong HH^{*}(A)\otimes HH^{*}(B)
\]
\end{teo}

\begin{proof}
By proposition \ref{tensoraction}, the action for $A\otimes B$ is given by the tensor product of the actions. Therefore,
$$
	\rho_a\otimes \rho_b: HH^*(A\otimes B; A\otimes B) \rightarrow HH_{n+m-*}(A\otimes B;A\otimes B)
$$
is an isomorphism. Then
$$
\Delta^{A\otimes B}= (\rho^{-1}_a\otimes \rho^{-1}_b) (B^A\otimes id + id\otimes B^B)( \rho_a\otimes \rho_b) = \Delta^{A}\otimes id + id\otimes \Delta^{B} \notag 
$$
\end{proof}

\section{BV-Algebra Structure on \texorpdfstring{\boldmath{$HH^{*}(\Z[\Z/n\Z])$}}{HH*(R[Z/nZ])}}

From now on, we assume that $A$ is $R\left[\Z /n \Z \right]\cong R [\sigma] /(\sigma^{n}-1)$ with $R$ a commutative ring. Since the Hochschild (co)homology of an associative algebra can be calculated using projective $A^{e}$-resolutions and the bar construction is not convenient to make explicit calculations, we are going to use the following 2-periodical resolution \cite{cyho}, \cite{Holm}.

\begin{pro}
The following is a $A^{e}$-projective resolution of $A$
\[
	\PP(A): \, \cdots \rightarrow A\otimes A \xrightarrow{d_2} A\otimes A \xrightarrow{d_1} A\otimes A \xrightarrow{\mu} A \rightarrow 0
\]
with 
\begin{align}
\mu (a\otimes b)&=ab \notag \\
d_{2k+1}(a\otimes b)&= (a\otimes b)(1\otimes \sigma - \sigma \otimes 1) \notag \\
d_{2k}(a\otimes b)&= (a\otimes b) \sum_{i=0} ^{n-1} \sigma^i \otimes \sigma^{n-i-1} \notag
\end{align}
\end{pro} 

\begin{proof}
First of all, notice that $A\otimes A\cong A^{e}$ as $A^{e}$-modules, so $A^2$ is $A^{e}$-free. From the definition, it follows that $d_{r}d_{r+1}=0$. Now, we are going to define the following $A$-right maps  
\[
	 \, \cdots \xleftarrow{\tilde{s}_3} A\otimes A \xleftarrow{\tilde{s}_2} A\otimes A \xleftarrow{\tilde{s}_1} A\otimes A \xleftarrow{\tilde{s}_0} A \leftarrow 0
\]
\begin{align}
	\tilde{s}_0 : A &\longrightarrow A^2, & \tilde{s}_0 (\sigma ^{i})&=1\otimes \sigma^{i} \notag \\
	\tilde{s}_{2k+1} : A^2 &\longrightarrow A^2, & \tilde{s}_{2k+1} (\sigma ^{i}\otimes 1)&= \begin{cases}
									-\displaystyle\sum_{j=0} ^{i-1} \sigma^j \otimes \sigma^{i-j-1} &\text{if $i\neq 0$} \notag\\
									0 &\text{if $i=0$}
								\end{cases} \notag \\
	\tilde{s}_{2k} : A^2 &\longrightarrow A^2, & \tilde{s}_{2k} (\sigma ^{i}\otimes 1)&= \begin{cases}
									1\otimes 1 &\text{if $i=n-1$} \notag\\
									0 &\text{otherwise}
								\end{cases} \notag
\end{align}
and by direct calculations, it follows that $\mu \tilde{s}_0= id$ and $d_{k+1}\tilde{s}_{k+1} + \tilde{s}_{k}d_k = id$ for all $k\geq 1$. Therefore, the complex is acyclic. 
\end{proof}

Tensoring this resolution by $A$ as $A^{e}$-modules and using the identification $A^2 \underset{A^e}\otimes A \cong A$, $((x\otimes y)\otimes a\mapsto yax)$, we obtain the complex
\begin{equation}\label{ho2Per}
\cdots  \rightarrow A \xrightarrow{n\sigma ^{n-1}} A \xrightarrow{0} A \xrightarrow{n\sigma ^{n-1}} A \xrightarrow{0} A
\end{equation}
Taking $Hom_{A^{e}}(- , A)$ of $\PP(A)$ and using the identification $Hom_{A^{e}}(A^2 , A)\cong A$, $(f \mapsto f(1\otimes 1))$, we obtain the complex
\begin{equation}\label{coh2Per}
A \xrightarrow{0} A \xrightarrow{n\sigma ^{n-1}} A \xrightarrow{0} A \xrightarrow{n\sigma ^{n-1}} A \rightarrow \cdots
\end{equation}

Then
$$
	HH_i(A) = \begin{cases}
					A &\text{if $i=0$} \notag\\
					A/(n\sigma ^{n-1}A) &\text{if $i=2k+1$} \notag \\
					Ann(n\sigma ^{n-1}) &\text{if $i=2k$}
			  \end{cases} \notag
$$
$$
	HH^i(A) = \begin{cases}
					A &\text{if $i=0$} \notag\\
					Ann(n\sigma ^{n-1}) &\text{if $i=2k+1$} \notag \\
					A/(n\sigma ^{n-1}A) &\text{if $i=2k$}
			  \end{cases} \notag
$$

To calculate the algebraic structures of $HH^{*}(A;A)$, we use two chain maps between $\PP(A)$ and the normalized bar resolution $\BB(A)$
\begin{align}
	\psi_*:& \; \PP(A)\rightarrow \BB(A) \notag \\
 	\varphi_*:& \; \,\BB(A)\rightarrow \PP(A) \notag
\end{align}
which are homotopy equivalences.

The $A^{e}$-homomorphisms $\psi_*$ will be defined by 
\begin{align}
	&\psi_0 = id : A^{2} \longrightarrow A^{2} \notag \\
	&\psi_{r+1}: A^{2} \longrightarrow A\otimes \bar{A}^{r+1}\otimes A, \qquad \psi_{r+1}(1\otimes 1):= s_r \psi_{r} d_{r+1}(1\otimes 1) \notag
\end{align} 
By direct computations, it follows that 
\begin{align}
	&\psi_{2r}(1\otimes 1)= \displaystyle \sum_{0\leq i_1,\dots ,i_r \leq n-1} (-1)^r 1\otimes \sigma^{i_1}\otimes \sigma \otimes \sigma^{i_2}\otimes \cdots \otimes \sigma^{i_r}\otimes \sigma\otimes \sigma^{r(n-1)-\sum_{k=1} ^{r} i_k} \notag \\
	&\psi_{2r+1}(1\otimes 1)= \displaystyle \sum_{0\leq i_1,\dots ,i_r \leq n-1} (-1)^{r+1} 1\otimes \sigma \otimes \sigma^{i_1}\otimes  \cdots \otimes \sigma^{i_r}\otimes \sigma\otimes \sigma^{r(n-1)-\sum_{k=1} ^{r} i_k} \notag
\end{align}

And the $A^{e}$-homomorphisms $\varphi_*$ will be defined by 
\begin{align}
	&\varphi_0 = id : A^{2} \longrightarrow A^{2} \notag \\
	&\varphi_1: A\otimes \bar{A}\otimes A \longrightarrow A^{2}, \; \varphi_1 (1\otimes \sigma^{i}\otimes 1):= -\sum_{j=0} ^{i-1} \sigma^j \otimes \sigma^{i-j-1} \notag \\
	&\varphi_2: A\otimes \bar{A}^{2}\otimes A \longrightarrow A^{2}, \, \varphi_2 (1\otimes \sigma^{i}\otimes \sigma^{k}\otimes 1):= \begin{cases}
									-1\otimes \sigma^{i+k-n} &\text{if $i+k\geq n$} \notag\\
									0 &\text{otherwise}
								\end{cases} \notag  
\end{align} 
$\varphi_r: A\otimes \bar{A}^{r}\otimes A \longrightarrow A^{2}$ for $r>2$ 
\begin{equation}
\resizebox{\textwidth}{!}{$
\varphi_r(1 \otimes \sigma^{i_1}\otimes  \cdots \otimes \sigma^{i_r}\otimes 1)= \varphi_{r-2}(1 \otimes \sigma^{i_1}\otimes  \cdots \otimes \sigma^{i_{r-2}}\otimes 1)\cdot \varphi_2 (1\otimes \sigma^{i_{r-1}}\otimes \sigma^{i_r}\otimes 1) \notag $}
\end{equation}

By direct computations, it follows that 

\begin{align}
\varphi_{2r}(1 \otimes \sigma^{i_1}\otimes  \cdots \otimes &\sigma^{i_{2r}}\otimes 1)= \displaystyle \prod_{k=1}^{r}\varphi_2 (1\otimes \sigma^{i_{2k-1}}\otimes \sigma^{i_{2k}}\otimes 1) \notag \\
&= \begin{cases}
		(-1)^r \otimes \sigma^{\sum_{k=1} ^{2r} i_k-rn} &\text{if $i_{2k-1}+i_{2k}\geq n$ for $1\leq k \leq r$} \notag\\
		0 &\text{otherwise}
	\end{cases} \notag  
\end{align}
\begin{align}
\varphi_{2r+1}(1 &\otimes \sigma^{i_1}\otimes  \cdots \otimes \sigma^{i_{2r+1}}\otimes 1)= \varphi_1(1\otimes \sigma^{i_1}\otimes 1)\displaystyle \prod_{k=1}^{r}\varphi_2 (1\otimes \sigma^{i_{2k}}\otimes \sigma^{i_{2k+1}}\otimes 1) \notag \\
&= \begin{cases}
		(-1)^{r+1} \displaystyle \sum_{j=0}^{i_1 -1} \sigma^{j} \otimes \sigma^{\sum_{k=1} ^{2r+1} i_k-j-rn-1} &\text{if $i_{2k}+i_{2k+1}\geq n$ for $1\leq k \leq r$} \notag\\
		0 &\text{otherwise}
	\end{cases} \notag  
\end{align}

\begin{remk}
These two maps gives an equivalence of complexes. Moreover, 
$$
\varphi_* \psi_*=id \quad \text{and} \quad \psi_* \varphi_* \simeq id
$$
\end{remk}

\begin{pro}
Using the identifications
$$
A\otimes \bar{A}^*\otimes A \underset{A^{e}}\otimes A\cong  \bar{A}^*\otimes A \quad \text{and} \quad A^2 \underset{A^{e}}\otimes A\cong A
$$
the induced maps for $\psi_*$ and $\varphi_*$ are
\begin{align}
	&\bar{\psi}_{*}: A \longrightarrow \bar{A}^* \otimes A \notag \\
	&\bar{\psi}_{2r}(a)= \displaystyle \sum_{0\leq i_1,\dots ,i_r \leq n-1} (-1)^r \sigma^{i_1}\otimes \sigma\otimes \sigma^{i_2}\otimes \cdots \otimes \sigma^{i_r}\otimes \sigma \otimes \sigma^{r(n-1)-\sum_{k=1} ^{r} i_k} a \notag \\
	&\bar{\psi}_{2r+1}(a)= \displaystyle \sum_{0\leq i_1,\dots ,i_r \leq n-1} (-1)^{r+1} \sigma\otimes \sigma^{i_1}\otimes \cdots \otimes \sigma^{i_r}\otimes \sigma \otimes \sigma^{r(n-1)-\sum_{k=1} ^{r} i_k} a \notag \\
	&\bar{\varphi}_{*}:\bar{A}^* \otimes A \longrightarrow A \notag \\
	&\bar{\varphi}_{2r}(\sigma^{i_1}\otimes \cdots \otimes \sigma^{i_{2r}}\otimes a)= \begin{cases}
		(-1)^r \sigma^{\sum_{k=1} ^{2r} i_k-rn} a &\text{if $i_{2k-1}+i_{2k}\geq n$} \notag\\
		0 &\text{otherwise}
	\end{cases} \notag  \\
    &\bar{\varphi}_{2r+1}(\sigma^{i_1}\otimes \cdots \otimes \sigma^{i_{2r+1}}\otimes a)= \begin{cases}
		(-1)^{r+1} a i_1 \sigma^{\sum_{k=1} ^{2r+1} i_k-rn-1} &\text{if $i_{2k}+i_{2k+1}\geq n$} \notag\\
		0 &\text{otherwise}
	\end{cases} \notag 
\end{align}
where $1\leq k \leq r$.
\end{pro}

\begin{pro}
Using the identifications
$$
Hom_{_{A^{e}}} (A \otimes \bar{A}^* \otimes A, A)\cong  Hom (\bar{A}^*, A) \quad \text{and} \quad Hom_{A^{e}}(A^2 , A)\cong A
$$
the induced maps for $\psi_*$ and $\varphi_*$ are
\begin{align}
	&\bar{\psi}_{r}^{*}: Hom (\bar{A}^r, A) \longrightarrow A \notag \\
	&\bar{\psi}_{2r}^{*}(f)= \displaystyle \sum_{0\leq i_1,\dots ,i_r \leq n-1} (-1)^r f(\sigma^{i_1}, \sigma, \sigma^{i_2}, \dots ,\sigma^{i_r}, \sigma) \sigma^{r(n-1)-\sum_{k=1} ^{r} i_k} \notag \\
	&\bar{\psi}_{2r+1}^{*}(f)= \displaystyle \sum_{0\leq i_1,\dots ,i_r \leq n-1} (-1)^{r+1} f(\sigma, \sigma^{i_1}, \dots , \sigma^{i_r}, \sigma ) \sigma^{r(n-1)-\sum_{k=1} ^{r} i_k} \notag \\
	&\bar{\varphi}_{r}^{*}:A \longrightarrow Hom_{_{R}} (\bar{A}^r, A) \notag \\
	&\bar{\varphi}_{2r}^{*}(a)(\sigma^{i_1}, \dots , \sigma^{i_{2r}})= \begin{cases}
		(-1)^r a \sigma^{\sum_{k=1} ^{2r} i_k-rn} &\text{if $i_{2k-1}+i_{2k}\geq n$} \notag\\
		0 &\text{otherwise}
	\end{cases} \notag  \\
    &\bar{\varphi}_{2r+1}^{*}(a)(\sigma^{i_1}, \dots , \sigma^{i_{2r+1}})= \begin{cases}
		(-1)^{r+1} a i_1 \sigma^{\sum_{k=1} ^{2r+1} i_k-rn-1} &\text{if $i_{2k}+i_{2k+1}\geq n$} \notag\\
		0 &\text{otherwise}
	\end{cases} \notag 
\end{align}
where $1\leq k \leq r$.
\end{pro}

%

%
%
%

\subsection{Cup Product and Cohomology Ring}

\begin{lem}\label{cupeven}
Let $R$ be a commutative ring. Then the cup product on the even Hochschild cohomology of $A=R [\sigma] /(\sigma^{n}-1)$ is induced by multiplication in $A$.
\end{lem}

\begin{proof}
Let $a\in HH^{2r}(A;A)$ and $b\in HH^{2s}(A;A)$. Then 
$$
\bar{\varphi}_{2r}^{*}(a) \smile \bar{\varphi}_{2s}^{*}(b)\in Hom(\bar{A}^{2(r+s)}, A)
$$ 
and

\begin{equation}
\resizebox{\textwidth}{!}{$
(\bar{\varphi}_{2r}^{*}(a) \smile \bar{\varphi}_{2s}^{*}(b))(\sigma^{i_1}, \dots , \sigma^{i_{2(r+s)}}) = (\bar{\varphi}_{2r}^{*}(a))(\sigma^{i_1}, \dots , \sigma^{i_{2r}})\cdot (\bar{\varphi}_{2s}^{*}(b))(\sigma^{i_{2r+1}}, \dots , \sigma^{i_{2(r+s)}})$} \notag
\end{equation}
\begin{align}
&= \begin{cases}
		(-1)^{r+s} ab \sigma^{\sum_{k=1} ^{2(r+s)} i_k-(r+s)n} &\text{if $i_{2k-1}+i_{2k}\geq n$ for $1\leq k \leq r+s$} \notag\\
		0 &\text{otherwise}
	\end{cases} \notag  \\
&=	(\bar{\varphi}_{2(r+s)}^{*}(ab))(\sigma^{i_1}, \dots , \sigma^{i_{2(r+s)}})\notag
\end{align}
Since $\bar{\psi}^{*}\bar{\varphi}^{*}=id$, then the cup product is induced by multiplication in $A$.
\end{proof}

\begin{lem}\label{cupodd}
$\smile : HH^{i}(A;A)\otimes HH^{j}(A;A)\rightarrow HH^{i+j}(A;A)$ is induced by multiplication if $i$ or $j$ is even, and by the formula 
$$
	a\smile b = - \frac{(n-1)n}{2} ab\sigma^{n-2}
$$
if $i$ and $j$ are odd.
\end{lem}

\begin{proof}
Let $a\in HH^{2r+1}(A;A)$ and $b\in HH^{2s}(A;A)$. Then 
$$
\bar{\varphi}_{2r+1}^{*}(a) \smile \bar{\varphi}_{2s}^{*}(b)\in Hom (\bar{A}^{2(r+s)+1}, A)
$$ 
and

\begin{equation}
\resizebox{\textwidth}{!}{$
(\bar{\varphi}_{2r+1}^{*}(a) \smile \bar{\varphi}_{2s}^{*}(b))(\sigma^{i_1}, \dots , \sigma^{i_{2(r+s)+1}}) = (\bar{\varphi}_{2r+1}^{*}(a))(\sigma^{i_1}, \dots , \sigma^{i_{2r+1}})\cdot (\bar{\varphi}_{2s}^{*}(b))(\sigma^{i_{2r+2}}, \dots , \sigma^{i_{2(r+s)+1}})$} \notag
\end{equation}
\begin{align}
&= \begin{cases}
		(-1)^{r+s+1} ab i_1 \sigma^{\sum_{k=1} ^{2(r+s)+1} i_k-(r+s)n-1} &\text{if $i_{2k}+i_{2k+1}\geq n$ for $1\leq k \leq r+s$} \notag\\
		0 &\text{otherwise}
	\end{cases} \notag  \\
&=	(\bar{\varphi}_{2(r+s)+1}^{*}(ab))(\sigma^{i_1}, \dots , \sigma^{i_{2(r+s)+1}})\notag
\end{align}
Then the cup product is induced by multiplication in $A$ if $\left|a\right|$ or $\left|b\right|$ is even.

Assume now that $a\in HH^{2r+1}(A;A)$ and $b\in HH^{2s+1}(A;A)$. Then 
$$
\bar{\varphi}_{2r+1}^{*}(a) \smile \bar{\varphi}_{2s}^{*}(b)\in Hom (\bar{A}^{2(r+s+1)}, A)
$$ 
and

\begin{equation}
\resizebox{\textwidth}{!}{$
(\bar{\varphi}_{2r+1}^{*}(a) \smile \bar{\varphi}_{2s+1}^{*}(b))(\sigma^{i_1}, \dots , \sigma^{i_{2(r+s+1)}}) = (\bar{\varphi}_{2r+1}^{*}(a))(\sigma^{i_1}, \dots , \sigma^{i_{2r+1}})\cdot (\bar{\varphi}_{2s+1}^{*}(b))(\sigma^{i_{2r+2}}, \dots , \sigma^{i_{2(r+s+1)}})$} \notag
\end{equation}
\begin{align}
&= \begin{cases}
		(-1)^{r+s} ab i_1 i_{2r+2} \sigma^{\sum_{k=1} ^{2(r+s+1)} i_k-(r+s)n-2} &\text{\footnotesize{if $i_{2k}+i_{2k+1}\geq n$ for $1\leq k \leq r$ and}} \notag\\
		 &\text{\footnotesize{$i_{2k-1}+i_{2k}\geq n$ for $r+2\leq k \leq r+s+1$}} \notag\\ 
		0 &\text{otherwise}
	\end{cases} \notag 
\end{align} 
Applying $\bar{\psi}_{2(r+s+1)}^{*}$ to $f=\bar{\varphi}_{2r+1}^{*}(a) \smile \bar{\varphi}_{2s+1}^{*}(b)$, we have
\begin{align}
&\bar{\psi}_{2(r+s+1)}^{*}(f) \notag \\
&= \displaystyle \sum_{0\leq i_1,\dots ,i_{r+s+1} \leq n-1} (-1)^{r+s+1} f(\sigma^{i_1}, \sigma, \sigma^{i_2}, \dots ,\sigma^{i_{r+s+1}}, \sigma) \sigma^{(r+s+1)(n-1)-\sum_{k=1} ^{r+s+1} i_k} \notag \\ 
&= - \displaystyle \sum_{i_1 =1}^{n-1} ab i_1\sigma^{n-2} = - \frac{(n-1)n}{2} ab\sigma^{n-2} \notag
\end{align}
Therefore, if $char(R)=p>0$ and $n=mp$, we have
\begin{equation}\label{cupp}
	a\smile b = \begin{cases}
	 mab\sigma^{2m-2} &\text{if $p=2$} \\
	 0 &\text{if $p\neq 2$}  
		\end{cases} 
\end{equation}
\end{proof}

From these results, now we can described the cohomology ring. 

\begin{teo}\label{ringeven}
Let $R$ be a commutative ring and $A=R [\sigma] /(\sigma^{n}-1)$. Then
$$
	HH^{2*}(A;A)=R [x,z]/(x^{n}-1,nz)
$$ 
where $x\in HH^{0}(A;A)$ and $z\in HH^{2}(A;A)$. 
\end{teo}
\begin{proof}
Consider $x\in HH^{0}(A;A)$ to be the coset $[\sigma]\in A$ and $z\in HH^{2}(A;A)$ the coset $[1]\in A$. By lemma \ref{cupeven}, the cup product for even degrees is induced by multiplication in $A$. Then $x$ generates $HH^{0}(A;A)$, and $HH^{2}(A;A)$ is generated by $z$ and $HH^{0}(A;A)$. In higher degrees $HH^{2i}(A;A)$ is generated by $z^{i}$ and $HH^{0}(A;A)$. The relations are given by $x^{n}-1=0$ and $nx^{n-1}z=0$.
\end{proof}

\begin{cor}\label{integraldomain}
Let $R$ be an integral domain with $char(R)\nmid n$. Then 	
$$
	HH^{*}(A;A)=R [x,z]/(x^{n}-1,nz)
$$
\end{cor}

\begin{proof}
Since $R$ is an integral domain with $char(R)\nmid n$, we have $Ann(n\sigma^{n-1})=0$.
\end{proof}

\begin{cor}
Let $R$ be a commutative ring such that $n\in R^{*}$. Then 	
$$
	HH^{*}(A;A)=R [x]/(x^{n}-1)=A
$$
\end{cor}

\begin{proof}
Since $n\in R^{*}$, we have $Ann(n\sigma^{n-1})=Ann(\sigma^{n-1})=0$, and $x^{n-1}z=0$ implies that $z=0$.
\end{proof}

\begin{teo}\label{algebraHHZp}
Let $R$ be a commutative ring with $char(R)=p>0$ and $A=R [\sigma] /(\sigma^{n}-1)$ with $n=mp$. If $p\neq 2$, or $p=2$ and $m$ is even. Then 
$$
HH^{*}(A;A)=R [x,y,z]/(x^{n}-1,y^{2}) 
$$
If $p=2$ and $m$ is odd. Then 
$$
HH^{*}(A;A)=R [x,y,z]/(x^{n}-1,y^{2}-x^{n-2}z) 
$$
where $x\in HH^{0}(A;A)$, $y\in HH^{1}(A;A)$ and $z\in HH^{2}(A;A)$. 
\end{teo}

\begin{proof}
By theorem \ref{ringeven}, we know that 
$$
	HH^{2*}(A;A)=R [x,z]/(x^{n}-1)
$$
Consider $y\in HH^{1}(A;A)$ to be the coset $[1]\in A$. Since cup product of an odd degree cohomology class and an even degree cohomology class is induced by multiplication in $A$, $HH^{1}(A;A)$ is generated by $y$ and $HH^{0}(A;A)$. By (\ref{cupp}), for $p\neq 2$, or $p=2$ and $m$ even the cup product in odd degrees is zero. Therefore, $y^{2}=0$ and we have
$$
	HH^{*}(A;A)=R [x,y,z]/(x^{n}-1,y^{2})
$$
For $p=2$ and $m$ odd, $y^{2}$ is the coset $[\sigma^{n-2}]\in A$ then  $y^{2}-x^{n-2}z=0$, and 
$$
	HH^{*}(A;A)=R [x,y,z]/(x^{n}-1,y^{2}-x^{n-2}z)
$$ 
\end{proof}

\begin{remk}
These calculations agree with the ones presented in \cite{SolotarCibils} and \cite{Holm}.
\end{remk}

%
%
%

\subsection{BV-Algebra Structure}

\begin{teo}\label{BVZn}
Let $R$ be an integral domain with $char(R)\nmid n$ and $A=R [\sigma] /(\sigma^{n}-1)$. Then the canonical Frobenius form of the group ring induces a BV-algebra structure on $HH^{*}(A;A)$ given by  
\begin{align}
	HH^{*}(A;A)&=R [x,z]/(x^{n}-1,nz) \notag \\
		\Delta(a)&=0 \quad \forall a\in HH^{*}(A;A) \notag
\end{align}	
\end{teo}

\begin{proof}
By corollary \ref{integraldomain}, we have $HH^{*}(A;A)=HH^{2*}(A;A)$. Then $\Delta(a)=0$ for all $a\in HH^{*}(A;A)$. However, this can be proved directly from the definition of $\Delta$, and the fact that in a BV-algebra we have the following equation
\begin{align}\label{Dabc}
	\Delta(abc)=&\,\Delta(ab)c+(-1)^{|a|}a\Delta(bc)+(-1)^{(|a|-1)|b|}b\Delta(ac) \\
	&-\Delta(a)bc-(-1)^{|a|}a\Delta(b)c-(-1)^{|a|+|b|}ab\Delta(c) \notag
\end{align}
Since the BV-operator is defined over the bar complex, we need the cochains that represent the generators. The class $x$ is represented by the cochain
$$
	\bar{\varphi}^{*}_{0}(\sigma)(1)=\sigma
$$  
and the class $z$ by
$$
	 \bar{\varphi}^{*}_{2}(1)(\sigma^{i},\sigma^{k})= \begin{cases}
	 -\sigma^{i+k-n} &\text{if $i+k\geq n$} \\
	 0 &\text{otherwise}  
		\end{cases} 
$$
Now, taking $\lbrace 1,\sigma,\dots ,\sigma^{n-1}\rbrace$ as a basis for $A$ and $\lbrace 1,\sigma^{n-1},\dots ,\sigma\rbrace$ as the dual basis induced by the canonical Frobenius form (\ref{frob}), we have
\begin{align*}
	\Aboxed{&\Delta(x)=\, \, 0} \; \text{ by degree.}  \notag \\	
	\Delta(\bar{\varphi}^{*}_{2}(1))(\sigma^{i})=\, &\sum_{k=0} ^{n-1} \langle 1,\bar{\varphi}^{*}_{2}(1)(\sigma^{k},\sigma^{i})\rangle \sigma^{n-k} - \sum_{k=0} ^{n-1} \langle 1,\bar{z}(\sigma^{i},\sigma^{k})\rangle \sigma^{n-k} \notag \\
	\Aboxed{&\Delta(z) =\; 0} \notag \\
	\Delta(\bar{\varphi}^{*}_{2}(1)\smile \bar{\varphi}^{*}_{0}(\sigma))(\sigma^{i})=\, &\sum_{k=0} ^{n-1} \langle 1,\bar{\varphi}^{*}_{2}(\sigma)(\sigma^{k},\sigma^{i})\rangle \sigma^{n-k} - \sum_{k=0} ^{n-1} \langle 1,\bar{z}\bar{x}(\sigma^{i},\sigma^{k})\rangle \sigma^{n-k} \notag \\
	\Aboxed{&\Delta(zx) =\; 0} \notag \\
	\Delta(\left(\bar{\varphi}^{*}_{2}(1)\right)^{2})(\sigma^{i}, \sigma^{j}, \sigma^{h})=\, &\sum_{k=0} ^{n-1} \langle 1,\bar{\varphi}^{*}_{4}(1)(\sigma^{k},\sigma^{i},\sigma^{j},\sigma^{h})\rangle \sigma^{n-k} - \sum_{k=0} ^{n-1} \langle 1,\bar{\varphi}^{*}_{4}(1)(\sigma^{i},\sigma^{j},\sigma^{h},\sigma^{k})\rangle \sigma^{n-k}\notag \\
	+&\sum_{k=0} ^{n-1} \langle 1,\bar{\varphi}^{*}_{4}(1)(\sigma^{j},\sigma^{h},\sigma^{k},\sigma^{i})\rangle \sigma^{n-k} - \sum_{k=0} ^{n-1} \langle 1,\bar{\varphi}^{*}_{4}(1)(\sigma^{h},\sigma^{k},\sigma^{i},\sigma^{j})\rangle \sigma^{n-k}\notag \\
	\Aboxed{&\Delta(z^{2}) =\, \, 0} \notag 	
\end{align*} 
In the last case, $\langle 1,\cdot \rangle\neq 0$ only if $k+i+j+h-2n=n$, i.e., $k+i+j+h=3n$, but also $k+i,j+h,i+j,h+k\geq n$. Therefore, all the coefficients are zero. Using equation (\ref{Dabc}) and induction on powers of $x$ and $z$, we have $\Delta(a)=0$ for all $a\in HH^{*}(A;A)$.
\end{proof}

\begin{teo}\label{BVZncharp}
Let $R$ be a commutative ring with $char(R)=p>0$ and $A=R [\sigma] /(\sigma^{n}-1)$ with $n=mp$. If $p\neq 2$, or $p=2$ and $m$ is even. Then the canonical Frobenius form of the group ring induces a BV-algebra structure on $HH^{*}(A;A)$ given by 
\begin{align}
HH^{*}(A;A)&=R [x,y,z]/(x^{n}-1,y^{2}) \notag \\
\Delta (z^{k}x^{l}) &=0  \notag \\
\Delta (z^{k}yx^{l}) &= (l-1)z^{k}x^{l-1} \notag 
\end{align}
If $p=2$ and $m$ is odd. Then as a BV-algebra
\begin{align}
HH^{*}(A;A)&=R [x,y,z]/(x^{n}-1,y^{2}-x^{n-2}z) \notag \\
\Delta (z^{k}x^{l}) &=0  \notag \\
\Delta (z^{k}yx^{l}) &= (l-1)z^{k}x^{l-1} \notag 
\end{align}
where $x\in HH^{0}(A;A)$, $y\in HH^{1}(A;A)$ and $z\in HH^{2}(A;A)$. 
\end{teo}

\begin{proof}
As in the previous theorem, we need the cochains that represent the generators. The class $y$ is represented by the cochain 
$$
	\bar{\varphi}^{*}_{1}(1)(\sigma^{i})=-i\sigma^{i-1}
$$ 
\begin{align*}
	\Delta(\bar{\varphi}^{*}_{1}(1))(1)=\, &\sum_{j=0} ^{n-1} \langle 1,\bar{\varphi}^{*}_{1}(1)(\sigma^{j})\rangle \sigma^{n-j} = \sum_{i=0} ^{n-1} - j\langle 1,\sigma^{j-1}\rangle \sigma^{n-j} \notag \\	
	= & - \sigma^{n-1} = -\bar{\varphi}^{*}_{0}(\sigma^{n-1})(1) \notag \\
	\Aboxed{&\Delta(y)=\, -x^{n-1}} \notag \\	
	\Delta(\bar{\varphi}^{*}_{0}(\sigma)\smile \bar{\varphi}^{*}_{1}(1))(1)=\, &\sum_{j=0} ^{n-1} \langle 1,\bar{\varphi}^{*}_{1}(\sigma)(\sigma^{j})\rangle \sigma^{n-j} = \sum_{i=0} ^{n-1} - j\langle 1,\sigma^{j}\rangle \sigma^{n-j} \notag \\
	\Aboxed{&\Delta(xy)=\, 0} \notag \\
	\Delta(\bar{\varphi}^{*}_{2}(1)\smile \bar{\varphi}^{*}_{1}(1))(\sigma^{i}, \sigma^{j})=\, &\sum_{k=0} ^{n-1} \langle 1,\bar{\varphi}^{*}_{3}(1)(\sigma^{k},\sigma^{i},\sigma^{j})\rangle \sigma^{n-k} \notag \\
	+ &\sum_{k=0} ^{n-1} \langle 1,\bar{\varphi}^{*}_{3}(1)(\sigma^{i},\sigma^{j},\sigma^{k})\rangle \sigma^{n-k}\notag \\
	+&\sum_{k=0} ^{n-1} \langle 1,\bar{\varphi}^{*}_{3}(1)(\sigma^{j},\sigma^{k},\sigma^{i})\rangle \sigma^{n-k} \notag 		
\end{align*}
If $i+j<n$, $\langle 1,\cdot \rangle= 0$ for all $k$. When $i+j\geq n$, $\langle 1,\cdot \rangle\neq 0$ only if $k+i+j-1-n=n$, i.e., $k=2n+1-(i+j)$. Therefore,
\begin{align*}
	\Delta(\bar{\varphi}^{*}_{3}(1))(\sigma^{i}, \sigma^{j})=\, &(2n+1-(i+j)) \sigma^{i+j-n-1} + i\sigma^{i+j-n-1} + j\sigma^{i+j-n-1}\notag \\
	= &(2n+1) \sigma^{i+j-n-1} = \sigma^{i+j-n-1} = -\bar{\varphi}^{*}_{2}(\sigma^{n-1})(\sigma^{i}, \sigma^{j}) \tag{$char(R)=p$ and $n=mp$}	\\
	\Aboxed{&\Delta(zy)=\, -zx^{n-1}} \notag	
\end{align*}  
Using equation (\ref{Dabc}) and induction on powers of $x$, $y$ and $z$, we have
\begin{align}
\Delta (z^{k}x^{l}) &=0  \notag \\
\Delta (z^{k}yx^{l}) &= (l-1)z^{k}x^{l-1} \notag 
\end{align}
\end{proof}

Since in a BV-algebra, the Gerstenhaber bracket is defined by the following equation 
\begin{align}\label{aDb}
	\lbrace a,b\rbrace = -(-1)^{\left|a\right|}( \Delta(ab) - \Delta(a)b - (-1)^{|a|}a\Delta(b))
\end{align}
It follows that

\begin{cor}
Let $A=R [\sigma] /(\sigma^{n}-1)$ with $R$ an integral domain and $char(R)\nmid n$. The Gerstenhaber bracket on $HH^{*}(A;A)$ is given by
$$
	\lbrace a, b\rbrace =0 \quad \forall a,b\in HH^{*}(A;A) 
$$	
\end{cor}

\begin{cor}
Let $R$ be a commutative ring with $char(R)=p>0$ and $A=R [\sigma] /(\sigma^{n}-1)$ with $n=mp$. The Gerstenhaber bracket on $HH^{*}(A;A)$ is given by
\begin{align}
\lbrace z^{k_1}x^{l_1},z^{k_2}x^{l_2}\rbrace &= 0  \notag \\ 
	\lbrace z^{k_1}x^{l_1},z^{k_2}yx^{l_2}\rbrace &= -l_1 z^{k_1 + k_2}x^{l_1 + l_2 -1}  \notag \\
	\lbrace z^{k_1}yx^{l_1}, z^{k_2}yx^{l_2}\rbrace &= (l_2 - l_1)z^{k_1 + k_2}yx^{l_1 + l_2 -1}  \notag
\end{align}
\end{cor}
\begin{remk}
These calculations agree with the Gerstenhaber bracket presented in \cite{Sanchez} and \cite{AltBracket}.
\end{remk}

For the cyclic group of order $p$ prime, $\Z /p \Z$, we have that the group ring $\Fp\left[\Z /p \Z \right]= \Fp [\sigma] /(\sigma^{p}-1)$  is naturally isomorphic, as algebra, to a truncated polynomial ring $ \Fp [x] /(x^{p})$. In \cite{AADD}, the authors transfer the canonical Frobenius form of the group ring to the truncated polynomial ring and get the following Frobenius form
$$
\varepsilon \left( \sum_{i=0} ^{p-1} \alpha_i x^i \right) = \sum_{i=0} ^{p-1} (-1)^{i}\alpha_i    
$$

Using this Frobenius form, the BV-algebra structure is given by  

\begin{teo}
Let $A=\Fp [x] /(x^{p})$ with $p$ an odd prime. Then the canonical Frobenius form of the group ring induces a BV-algebra structure on $HH^{*}(A;A)$ given by
\[
	HH^{*}(A;A)=\Fp [x,v,t]/(x^{p},v^{2})
\]
\begin{align}
&\Tilde{\Delta} (t^{l}x^{k}) =0  \notag \\
&\Tilde{\Delta} (t^{k}vx^{2l}) = 2lt^{k}x^{2l-1} + \sum_{i=2l} ^{p-1} (-1)^{i+1}t^{k}x^{i} \notag \\
&\Tilde{\Delta} (t^{k}vx^{2l+1}) = (2l+1)t^{k}x^{2l} + \sum_{i=2l+1} ^{p-1} (-1)^{i}t^{k}x^{i} \notag 
\end{align}

where $x\in HH^{0}(A;A)$, $v\in HH^{1}(A;A)$ and $t\in HH^{2}(A;A)$.
\end{teo}

\begin{cor}
There is an isomorphism of BV-algebras
$$
	\phi: HH^{*}(\Fp [x] /(x^{p}); \Fp [x] /(x^{p}))\overset{\cong}{\longrightarrow} HH^{*}(\Fp\left[\Z /p \Z \right];\Fp\left[\Z /p \Z \right])
$$
\end{cor}

\begin{proof}
The isomorphism $\phi$ is defined as follows
$$
	\phi(x)=x-1, \qquad \phi(v)=y \qquad \text{and} \qquad \phi(t)=z
$$
It is clear that $\phi$ is a ring isomorphism. To verify that it is an isomorphism of BV-algebras, we need to check that $\phi \Tilde{\Delta}=\Delta \phi$.
\begin{itemize}
\item $\phi \Tilde{\Delta}(x)=\phi(0)=0=\Delta (x-1)=\Delta \phi(x)$.
\item $\Tilde{\Delta} (v) = \sum_{i=0} ^{p-1} (-1)^{i+1}x^{i}$, then
\begin{align}
   \phi\Tilde{\Delta} (v) &= \sum_{i=0} ^{p-1} (-1)^{i+1}(x-1)^{i} = \sum_{i=0} ^{p-1} \sum_{k=0}^{i}(-1)^{k+1} \binom {i}{k}x^{k} \notag \\
   &= \sum_{k=0} ^{p-1} (-1)^{k+1}\sum_{i=k}^{p-1} \binom {i}{k}x^{k} \notag \\
   &\equiv -x^{p-1} \tag{mod $p$} \\
   &=\Delta \phi(v) \notag
\end{align} 
the equivalence module $p$ is due to 
$$
	\sum_{i=k}^{p-1} \binom {i}{k} = \binom {p}{k+1} \equiv 0 \qquad \text{(mod $p$)}
$$
for $k+1\neq 0$ or $k+1\neq p$.
\item $\phi \Tilde{\Delta}(t)=\phi(0)=0=\Delta (z)=\Delta \phi(t)$.
\item $\phi \Tilde{\Delta}(t^{2})=\phi(0)=0=\Delta (z^{2})=\Delta \phi(t^{2})$.
\item $\phi \Tilde{\Delta}(tx)=\phi(0)=0=\Delta (zx)-\Delta(z)=\Delta(z(x-1))=\Delta \phi(tx)$.
\item $\Tilde{\Delta} (tv) = \sum_{i=0} ^{p-1} (-1)^{i+1}tx^{i}$, then
\begin{align}
   \phi\Tilde{\Delta} (tv) &= \sum_{i=0} ^{p-1} (-1)^{i+1}z(x-1)^{i} = z\sum_{i=0} ^{p-1} \sum_{k=0}^{i}(-1)^{k+1} \binom {i}{k}x^{k} \notag \\
   &= z\sum_{k=0} ^{p-1} (-1)^{k+1}\sum_{i=k}^{p-1} \binom {i}{k}x^{k} \notag \\
   &\equiv -zx^{p-1} \tag{mod $p$} \\
   &= \Delta \phi(tv) \notag
\end{align} 
\item $\Tilde{\Delta} (vx) = \sum_{i=0} ^{p-1} (-1)^{i}x^{i}$, then
\begin{align}
   \phi\Tilde{\Delta} (vx) &= \sum_{i=0} ^{p-1} (-1)^{i}(x-1)^{i} = \sum_{i=0} ^{p-1} \sum_{k=0}^{i}(-1)^{k} \binom {i}{k}x^{k} \notag \\
   &= \sum_{k=0} ^{p-1} (-1)^{k}\sum_{i=k}^{p-1} \binom {i}{k}x^{k} \notag \\
   &\equiv x^{p-1} \tag{mod $p$} \\
   &= \Delta(yx)-\Delta(y)=\Delta(y(x-1))=\Delta \phi(vx) \notag
\end{align} 
\end{itemize}

Since both are BV-algebras, formula (\ref{Dabc}) holds and $\phi \Tilde{\Delta}=\Delta \phi$.
\end{proof}

And for $p=2$, 

\begin{teo}
Let $A=\FF [x] /(x^{2})$. Then the canonical Frobenius form of the group ring induces a BV-algebra structure on $HH^{*}(A;A)$ given by
$$
	HH^{*}(A;A)=\FF [x,v,t]/(x^{2},v^{2}-t)\cong \Lambda(x)\otimes \FF [v] 
$$
$$
	\Tilde{\Delta}(v^{k}x^{l})=k(1+x)v^{k-1}
$$	
where $x\in HH^{0}(A;A)$, $v\in HH^{1}(A;A)$ with $|x|=0$ and $|v|=1$.
\end{teo}
 
\begin{cor}
There is an isomorphism of BV-algebras
$$
	\phi: HH^{*}(\FF [x] /(x^{2}); \FF [x] /(x^{2}))\overset{\cong}{\longrightarrow} HH^{*}(\FF\left[\Z /2 \Z \right];\FF\left[\Z /2 \Z \right])
$$
\end{cor} 

\begin{proof}
For $p=2$ and $n=2$, we have
$$
HH^{*}(\FF\left[\Z /2 \Z \right];\FF\left[\Z /2 \Z \right])=R [x,y,z]/(x^{2}-1,y^{2}-z)\cong \FF [y]\otimes \FF [x]/(x^{2}-1)
$$
and the BV-operator, $\Delta$, is given by
\begin{align}
\Delta (y^{k}x) &=0  \notag \\
\Delta (y^{k}) &= ky^{k-1}x \notag 
\end{align}
The isomorphism $\phi$ is defined as follows
$$
	\phi(x)=x-1 \qquad \text{and} \qquad \phi(v)=y 
$$
It is clear that $\phi$ is a ring isomorphism. Now,
$$
\phi \Tilde{\Delta}(v^{k}x^{l})=\phi(k(1+x)v^{k-1})=kxy^{k-1}=\Delta(y^{k}(x-1)^{l})=\Delta \phi(v^{k}x^{l})
$$
Therefore, $\phi$ is an isomorphism of BV-algebras.
\end{proof}


\section{BV-Algebra Structure for Finite Abelian Groups} 

Let $G$ be a finite abelian group. Then $G$ can be decomposed as follows 
$$
	G \cong \Z /p_1^{\alpha_1} \Z \oplus \Z /p_2^{\alpha_2} \Z \oplus \cdots \oplus \Z /p_k^{\alpha_k} \Z
$$
with the property that $\alpha_i \leq \alpha_{i+1}$, if $p_i = p_{i+1}$. Therefore,
$$
	R[G] \cong R[\Z /p_1^{\alpha_1} \Z] \otimes R[\Z /p_2^{\alpha_2} \Z] \otimes \cdots \otimes R[\Z /p_k^{\alpha_k} \Z]
$$
and by corollary \ref{tensorBViso}, we have 

\begin{teo}\label{BVfiniteabelian}
Let $R$ be a field and $G$ a finite abelian group. Then as BV-algebras 
$$
	HH^*(R[G];R[G]) \cong HH^*(R[\Z /p_1^{\alpha_1} \Z];R[\Z /p_1^{\alpha_1} \Z]) \otimes \cdots \otimes HH^*(R[\Z /p_k^{\alpha_k} \Z]; R[\Z /p_k^{\alpha_k} \Z])
$$
where the BV-structure for each factor is given by theorem \ref{BVZn} or \ref{BVZncharp}.
\end{teo}


\section{BV-Algebra Structure on \texorpdfstring{\boldmath{$HH^{*}(\Z[\Z/n\Z]\otimes\Z[\Z/m\Z])$}}{HH*(R[Z/nZxZ/mZ])}}

By theorem \ref{injection}, we have an injection of BV-algebras 
$$
	HH^{*}(\Z[\Z/n\Z])\otimes HH^{*}(\Z[\Z/m\Z])\hookrightarrow HH^{*}(\Z[\Z/n\Z]\otimes\Z[\Z/m\Z])
$$
where the BV-operator on the left hand side is trivial. Nevertheless, the BV-operator on the right hand side is highly non-trivial as follows 

\begin{teo}
Let $A=\Z[\Z/n\Z]$ and $B=\Z[\Z/m\Z]$ with $n=km$. Then as a BV-algebra,
\begin{align}
	HH^{*}(A\otimes B;A\otimes B)&\cong \frac{\Z[x,t,a,b,c]}{(x^n-1, t^m-1, na, mb, mc, c^2)}  \notag \\
\Delta(x^{i}t^{j}a^{l}b^{r}c^{s}) &= sx^{i-1}t^{j}a^{l}b^{r}((i-1)b-jka)  \notag 
\end{align}
in all cases except when $m$ is even and $k$ is odd, in which case we get
\begin{align}
	HH^{*}(A\otimes B;A\otimes B)&\cong \frac{\Z[x,t,a,b,c]}{(x^n-1, t^m-1, na, mb, mc, c^2-\frac{m}{2}x^{n-2}ab(b+ka))}  \notag \\
\Delta(x^{i}t^{j}a^{l}b^{r}c^{s}) &= sx^{i-1}t^{j}a^{l}b^{r}((i-1)b-jka)  \notag 
\end{align}
where $x,t\in HH^{0}(A\otimes B;A\otimes B)$, $a,b\in HH^{2}(A\otimes B;A\otimes B)$ and $c\in HH^{3}(A\otimes B;A\otimes B)$.
\end{teo}

\begin{proof}
By K\"{u}nneth Theorem, there is an isomorphim of modules 
\begin{align}
HH^{i}(A\otimes B; A\otimes B)\cong &\bigoplus_{r+s=i} HH^{r}(A; A)\otimes HH^{s}(B;B) \notag \\
&\bigoplus \bigoplus_{r+s=i+1} Tor_1 ^{\Z} (HH^{r}(A; A), HH^{s}(B;B)) \notag
\end{align}
Since,
$$
	HH^{*}(A;A)=\Z [x,a]/(x^{n}-1,na) \qquad \text{and} \qquad	HH^{*}(B;B)=\Z [t,b]/(t^{m}-1,mb)
$$
where $|x|=|t|=0$ and $|a|=|b|=2$. All Tor groups vanish except when $r$ and $s$ are both even. In order to calculate $Tor_1 ^{\Z} (A/nx^{n-1}A, B/mt^{m-1}B)$, we use the following $\Z$-projective resolution 
$$
  0 \rightarrow A \xrightarrow{nx^{n-1}} A \xrightarrow{} A/nx^{n-1}A \rightarrow 0 
$$
Applying $-\otimes B/mt^{m-1}B$, we get 
$$
  0 \rightarrow A\otimes B/mt^{m-1}B \xrightarrow{nx^{n-1}\otimes id} A \otimes B/mt^{m-1}B \rightarrow 0 
$$ 
Thus, 
\begin{align}
Tor_1 ^{\Z} (A/nx^{n-1}A, B/mt^{m-1}B)&=Ker(nx^{n-1}\otimes id)=Ker(kx^{n-1}\otimes m\cdot \,) \notag \\
&= A\otimes B/mt^{m-1}B \notag
\end{align}
Therefore,
$$
	HH^i(A\otimes B) = \begin{cases}
					A\otimes B &\text{if $i=0$} \notag\\
					0 &\text{if $i=1$} \notag \\
					A/nx^{n-1}A\otimes B \oplus A\otimes B/mt^{m-1}B &\text{if $i=2j$} \notag \\
					\quad \oplus \displaystyle\bigoplus_{l=1}^{j-1} A/nx^{n-1}A\otimes B/mt^{m-1}B \notag \\
					\displaystyle\bigoplus_{l=1}^{j} A\otimes B/mt^{m-1}B &\text{if $i=2j+1$} \notag \\
			  \end{cases} \notag
$$
Since $$HH^{*}(A; A)\otimes HH^{*}(B;B) \hookrightarrow HH^{*}(A\otimes B; A\otimes B)$$ is an injection of BV-algebras, we only need to find a generator for the odd dimensions. Let $c\in HH^{3}(A\otimes B; A\otimes B)$ to be the coset $[1\otimes 1]\in A\otimes B/mt^{m-1}B$. Consider  
$$
	c = 1\otimes 1 - kx^{n-1}\otimes t
$$
to be the representative of $[1\otimes 1]$ in the total complex which calculate the Tor group. In the tensor product of the bar resolutions, $c$ is represented by
$$
 c=\overline{\varphi_{_A}}_{1}^{*}(1)\otimes \overline{\varphi_{_B}}_{2}^{*}(1)-k \overline{\varphi_{_A}}_{2}^{*}(x^{n-1})\otimes \overline{\varphi_{_B}}_{1}^{*}(t)
$$ 
Since $HH^{2*}(A\otimes B;A\otimes B)\cong HH^{2*}(A;A)\otimes HH^{2*}(B;B)$. Let $y\in HH^{2i}(A;A)$ and $z\in HH^{2j}(B;B)$. Then 
\begin{align}
 y\otimes z \smile c \Rightarrow & \; (\overline{\varphi_{_A}}_{2i}^{*}(y)\otimes \overline{\varphi_{_B}}_{2j}^{*}(z))\smile (\overline{\varphi_{_A}}_{1}^{*}(1)\otimes \overline{\varphi_{_B}}_{2}^{*}(1)-k \overline{\varphi_{_A}}_{2}^{*}(x^{n-1})\otimes \overline{\varphi_{_B}}_{1}^{*}(t)) \notag \\
 =& \; (\overline{\varphi_{_A}}_{2i}^{*}(y)\smile \overline{\varphi_{_A}}_{1}^{*}(1))\otimes (\overline{\varphi_{_B}}_{2j}^{*}(z))\smile \overline{\varphi_{_B}}_{2}^{*}(1)) \notag \\
 &-k (\overline{\varphi_{_A}}_{2i}^{*}(y)\smile \overline{\varphi_{_A}}_{2}^{*}(x^{n-1}))\otimes (\overline{\varphi_{_B}}_{2j}^{*}(z))\smile \overline{\varphi_{_B}}_{1}^{*}(t)) \notag \\ 
 =& \; (\overline{\varphi_{_A}}_{2i+1}^{*}(y)\otimes \overline{\varphi_{_B}}_{2(j+1)}^{*}(z)) - k (\overline{\varphi_{_A}}_{2(i+1)}^{*}(yx^{n-1})\otimes \overline{\varphi_{_B}}_{2j+1}^{*}(zt)) \notag 
\end{align} 
Applying $\bar{\psi_{_A}^{*}}\otimes \bar{\psi_{_B}^{*}}$, we have 
$$
	(y\otimes z)(1\otimes 1 - k x^{n-1}\otimes t)
$$
Therefore, $HH^{2k+3}(A\otimes B;A\otimes B)$ is generated by $HH^{2k}(A\otimes B;A\otimes B)$ and $c\in HH^{3}(A\otimes B;A\otimes B)$. Now, consider $x$ to be the coset $[x\otimes 1]\in HH^{0}(A;A)\otimes HH^{0}(B;B)$, $t$ to be the coset $[1\otimes t] \in HH^{0}(A;A)\otimes HH^{0}(B;B)$, $a$ to be the coset $[1\otimes 1]\in HH^{2}(A;A)\otimes HH^{0}(B;B)$ and $b$ to be the coset $[1\otimes 1] \in HH^{0}(A;A)\otimes HH^{2}(B;B)$. Notice that $x,t,a$ and $b$ generate  $HH^{2*}(A\otimes B;A\otimes B)$, and satisfy the relations $x^n-1=0$, $t^m-1=0$, $nx^{n-1}a=0$ and  $mt^{m-1}b=0$. Now,
\begin{equation}
\resizebox{\textwidth}{!}{$
\begin{aligned}
 c^2 =& \; (\overline{\varphi_{_A}}_{1}^{*}(1)\otimes \overline{\varphi_{_B}}_{2}^{*}(1)-k \overline{\varphi_{_A}}_{2}^{*}(x^{n-1})\otimes \overline{\varphi_{_B}}_{1}^{*}(t))\smile (\overline{\varphi_{_A}}_{1}^{*}(1)\otimes \overline{\varphi_{_B}}_{2}^{*}(1)-k \overline{\varphi_{_A}}_{2}^{*}(x^{n-1})\otimes \overline{\varphi_{_B}}_{1}^{*}(t)) \notag \\
 c^2 =& \; \overline{\varphi_{_A}}_{1}^{*}(1)\smile \overline{\varphi_{_A}}_{1}^{*}(1)\otimes \overline{\varphi_{_B}}_{2}^{*}(1)\smile \overline{\varphi_{_B}}_{2}^{*}(1) + k \overline{\varphi_{_A}}_{2}^{*}(x^{n-1})\smile \overline{\varphi_{_A}}_{1}^{*}(1)\otimes \overline{\varphi_{_B}}_{1}^{*}(t)\smile \overline{\varphi_{_B}}_{2}^{*}(1) \notag \\
 &- k \overline{\varphi_{_A}}_{1}^{*}(1)\smile \overline{\varphi_{_A}}_{2}^{*}(x^{n-1})\otimes \overline{\varphi_{_B}}_{2}^{*}(1)\smile \overline{\varphi_{_B}}_{1}^{*}(t) + k^2 \overline{\varphi_{_A}}_{2}^{*}(x^{n-1})\smile \overline{\varphi_{_A}}_{2}^{*}(x^{n-1}) \otimes \overline{\varphi_{_B}}_{1}^{*}(t)\smile \overline{\varphi_{_B}}_{1}^{*}(t) \notag \\
 c^2 =& -\frac{(n-1)n}{2} x^{n-2}\otimes 1 - \frac{(m-1)mk^2}{2} x^{n-2}\otimes 1 \notag  \notag \\
 c^2=& -\frac{(n-1)n}{2} x^{n-2}ab^2 - \frac{(m-1)mk^2}{2} x^{n-2}a^2b \notag 
\end{aligned}$}
\end{equation}
Thus,
\begin{itemize}
\item If $n$ and $m$ are odd then $c^2=0$.
\item If $n$ is even and $m$ is odd then $k$ is even and $c^2=0$.
\item If $m$ is even then $n$ is even and 
\begin{align}
c^2&= -\frac{km}{2}((km-1)x^{n-2}ab^2+(m-1)kx^{n-2}a^2b) \notag \\
c^2&=\frac{km}{2}x^{n-2}ab(b+ka) \notag 
\end{align}
\item If $k$ is even then $c^2=0$.
\item If $k$ is odd then $c^2=\frac{m}{2}x^{n-2}ab(b+ka)$.
\end{itemize}
Also, notice that $mc=0$. To sum up, as algebras 
$$
	HH^{*}(A\otimes B;A\otimes B)\cong \frac{\Z[x,t,a,b,c]}{(x^n-1, t^m-1, na, mb, mc, c^2)}  
$$	
in all cases except when $m$ is even and $k$ is odd, in which case we get
$$
	HH^{*}(A\otimes B;A\otimes B)\cong \frac{\Z[x,t,a,b,c]}{(x^n-1, t^m-1, na, mb, mc, c^2-\frac{m}{2}x^{n-2}ab(b+ka))}  
$$	

It only remains to calculate the BV-operator. Using theorem \ref{BVinj}, the BV-operator $\Delta^{A\otimes B}$ can be calculated using $\Delta^A$ and $\Delta^B$. Using the equations calculated before for $\Delta^A$ and $\Delta^B$ on the cochain level (\ref{BVZncharp}), we have 
\begin{align*}
\Delta(c)=&\Delta ( \overline{\varphi_{_A}}_{1}^{*}(1)\otimes \overline{\varphi_{_B}}_{2}^{*}(1)- k \overline{\varphi_{_A}}_{2}^{*}(x^{n-1})\otimes \overline{\varphi_{_B}}_{1}^{*}(t) ) \notag \\
=& \Delta^A (\overline{\varphi_{_A}}_{1}^{*}(1))\otimes\overline{\varphi_{_B}}_{2}^{*}(1) - \overline{\varphi_{_A}}_{1}^{*}(1)\otimes \Delta^B(\overline{\varphi_{_B}}_{2}^{*}(1)) \notag \\ 
&- k \Delta^A(\overline{\varphi_{_A}}_{2}^{*}(x^{n-1}))\otimes \overline{\varphi_{_B}}_{1}^{*}(t) - k \overline{\varphi_{_A}}_{2}^{*}(x^{n-1})\otimes \Delta^B(\overline{\varphi_{_B}}_{1}^{*}(t)) \notag \\
=& -\overline{\varphi_{_A}}_{0}^{*}(x^{n-1})\otimes \overline{\varphi_{_B}}_{2}^{*}(1) \notag \\
\Aboxed{&\Delta(c)=-x^{n-1}b} \notag 
\end{align*}
\begin{align*}
\Delta(xc)=&\Delta ( \overline{\varphi_{_A}}_{0}^{*}(x)\smile \overline{\varphi_{_A}}_{1}^{*}(1)\otimes \overline{\varphi_{_B}}_{0}^{*}(1)\smile \overline{\varphi_{_B}}_{2}^{*}(1) \notag \\
&- k \overline{\varphi_{_A}}_{0}^{*}(x)\smile \overline{\varphi_{_A}}_{2}^{*}(x^{n-1})\otimes \overline{\varphi_{_B}}_{0}^{*}(1)\smile \overline{\varphi_{_B}}_{1}^{*}(t) ) \notag \\
=& \Delta^A (\overline{\varphi_{_A}}_{1}^{*}(x))\otimes\overline{\varphi_{_B}}_{2}^{*}(1) - \overline{\varphi_{_A}}_{1}^{*}(x)\otimes \Delta^B(\overline{\varphi_{_B}}_{2}^{*}(1)) \notag \\ 
&- k \Delta^A(\overline{\varphi_{_A}}_{2}^{*}(1))\otimes \overline{\varphi_{_B}}_{1}^{*}(t) - k \overline{\varphi_{_A}}_{2}^{*}(1)\otimes \Delta^B(\overline{\varphi_{_B}}_{1}^{*}(t)) \notag \\
=& 0 \notag \\
\Aboxed{&\Delta(xc)=0} \notag 
\end{align*}
\begin{align*}
\Delta(tc)=&\Delta ( \overline{\varphi_{_A}}_{0}^{*}(1)\smile \overline{\varphi_{_A}}_{1}^{*}(1)\otimes \overline{\varphi_{_B}}_{0}^{*}(t)\smile \overline{\varphi_{_B}}_{2}^{*}(1) \notag \\
&- k \overline{\varphi_{_A}}_{0}^{*}(1)\smile \overline{\varphi_{_A}}_{2}^{*}(x^{n-1})\otimes \overline{\varphi_{_B}}_{0}^{*}(t)\smile \overline{\varphi_{_B}}_{1}^{*}(t) ) \notag \\
=& \Delta^A (\overline{\varphi_{_A}}_{1}^{*}(1))\otimes\overline{\varphi_{_B}}_{2}^{*}(t) - \overline{\varphi_{_A}}_{1}^{*}(1)\otimes \Delta^B(\overline{\varphi_{_B}}_{2}^{*}(t)) \notag \\ 
&- k \Delta^A(\overline{\varphi_{_A}}_{2}^{*}(x^{n-1}))\otimes \overline{\varphi_{_B}}_{1}^{*}(t^2) - k \overline{\varphi_{_A}}_{2}^{*}(x^{n-1})\otimes \Delta^B(\overline{\varphi_{_B}}_{1}^{*}(t^2)) \notag \\
=& - \overline{\varphi_{_A}}_{0}^{*}(x^{n-1})\otimes \overline{\varphi_{_B}}_{2}^{*}(t) - k \overline{\varphi_{_A}}_{2}^{*}(x^{n-1})\otimes \overline{\varphi_{_B}}_{0}^{*}(t) \notag \\
\Aboxed{&\Delta(tc)=-x^{n-1}t(b+ka)} \notag 
\end{align*}
\begin{align*}
\Delta(ac)=&\Delta ( \overline{\varphi_{_A}}_{2}^{*}(1)\smile \overline{\varphi_{_A}}_{1}^{*}(1)\otimes \overline{\varphi_{_B}}_{0}^{*}(1)\smile \overline{\varphi_{_B}}_{2}^{*}(1) \notag \\
&- k \overline{\varphi_{_A}}_{2}^{*}(1)\smile \overline{\varphi_{_A}}_{2}^{*}(x^{n-1})\otimes \overline{\varphi_{_B}}_{0}^{*}(1)\smile \overline{\varphi_{_B}}_{1}^{*}(t) ) \notag \\
=& \Delta^A (\overline{\varphi_{_A}}_{3}^{*}(1))\otimes\overline{\varphi_{_B}}_{2}^{*}(1) - \overline{\varphi_{_A}}_{3}^{*}(1)\otimes \Delta^B(\overline{\varphi_{_B}}_{2}^{*}(1)) \notag \\ 
&- k \Delta^A(\overline{\varphi_{_A}}_{4}^{*}(x^{n-1}))\otimes \overline{\varphi_{_B}}_{1}^{*}(t) - k \overline{\varphi_{_A}}_{4}^{*}(x^{n-1})\otimes \Delta^B(\overline{\varphi_{_B}}_{1}^{*}(t)) \notag \\
=& -(2n+1)\overline{\varphi_{_A}}_{2}^{*}(x^{n-1}) \otimes \overline{\varphi_{_B}}_{2}^{*}(1) \notag \\
\Aboxed{&\Delta(ac)=-x^{n-1}ab} \notag 
\end{align*}
\begin{align*}
\Delta(bc)=&\Delta ( \overline{\varphi_{_A}}_{0}^{*}(1)\smile \overline{\varphi_{_A}}_{1}^{*}(1)\otimes \overline{\varphi_{_B}}_{2}^{*}(1)\smile \overline{\varphi_{_B}}_{2}^{*}(1) \notag \\
&- k \overline{\varphi_{_A}}_{0}^{*}(1)\smile \overline{\varphi_{_A}}_{2}^{*}(x^{n-1})\otimes \overline{\varphi_{_B}}_{2}^{*}(1)\smile \overline{\varphi_{_B}}_{1}^{*}(t) ) \notag \\
=& \Delta^A (\overline{\varphi_{_A}}_{1}^{*}(1))\otimes\overline{\varphi_{_B}}_{4}^{*}(1) - \overline{\varphi_{_A}}_{1}^{*}(1)\otimes \Delta^B(\overline{\varphi_{_B}}_{4}^{*}(1)) \notag \\ 
&- k \Delta^A(\overline{\varphi_{_A}}_{2}^{*}(x^{n-1}))\otimes \overline{\varphi_{_B}}_{3}^{*}(t) - k \overline{\varphi_{_A}}_{2}^{*}(x^{n-1})\otimes \Delta^B(\overline{\varphi_{_B}}_{3}^{*}(t)) \notag \\
=& -\overline{\varphi_{_A}}_{0}^{*}(x^{n-1}) \otimes \overline{\varphi_{_B}}_{4}^{*}(1) + 2km \overline{\varphi_{_A}}_{2}^{*}(x^{n-1}) \otimes \overline{\varphi_{_B}}_{2}^{*}(1) \notag \\
\Aboxed{&\Delta(bc)=-x^{n-1}b^2} \notag 
\end{align*}
Using equation \ref{Dabc} and induction on powers of $x,t,a,b$ and $c$, we have
$$
\Delta(x^{i}t^{j}a^{l}b^{r}c^{s}) = sx^{i-1}t^{j}a^{l}b^{r}((i-1)b-jka)
$$
\end{proof}


\section{BV-Algebra Structure on \texorpdfstring{\boldmath{$HH^{*}(R[\Z^{k}])$}}{HH*(R[Zk])}}

From now on, we assume that $A$ is $R[\Z]\cong R [t,t^{-1}]$ with $R$ a commutative ring. 

\begin{pro}
The following is a $A^{e}$-projective resolution of $A$
\begin{equation}\label{Pres}
	\PP(A): 0 \rightarrow A\otimes A \xrightarrow{d_1} A\otimes A \xrightarrow{\mu} A \rightarrow 0
\end{equation}
with $\mu (a\otimes b)=ab$ and $d_{1}(a\otimes b)= (a\otimes b)(1\otimes t - t \otimes 1)$.
\end{pro} 

\begin{proof}
From the definition, it follows that $\mu d_{1}=0$. Now, we are going to define the following $A$-right maps  
\[
	 0 \leftarrow A\otimes A \xleftarrow{s_1} A\otimes A \xleftarrow{s_0} A \leftarrow 0
\]
\begin{align}
	s_0 (a)&=1\otimes a \notag \\
	s_1 (t^{i}\otimes 1)&= \begin{cases}
									-\displaystyle\sum_{j=0} ^{i-1} t^j \otimes t^{i-j-1} &\text{if $i\geq 1$} \notag\\
									0 &\text{if $i=0$} \notag \\
									\displaystyle\sum_{j=0} ^{-i-1} t^{-j-1} \otimes t^{i+j} &\text{if $i\leq -1$} \notag
								\end{cases} \notag 
\end{align}
By direct calculations, it follows that $\mu s_0= id$ and $d_{k+1}s_{k+1} + s_{k}d_{k} = id$ for all $k\geq 0$. Therefore, the complex is acyclic. 
\end{proof}

Tensoring this resolution by $A$ as $A^{e}$-modules, we obtain the complex
\begin{equation}\label{hoZ}
0 \rightarrow A \xrightarrow{0} A
\end{equation}
Taking $Hom_{A^{e}}(-, A)$ of $\PP(A)$, we obtain the complex
\begin{equation}\label{cohZ}
A \xrightarrow{0} A \rightarrow 0
\end{equation}

Then
$$
	HH_i(A;A) = HH^i(A;A) = \begin{cases}
					A &\text{if $i=0,1$} \notag\\
					0 &\text{otherwise} \notag 
			  \end{cases} \notag
$$

To calculate the cup product, we define $\Delta_{\PP(A)}: \PP(A) \longrightarrow \PP(A)\underset{A}\otimes \PP(A)$ as follows
\begin{align}\label{diagA}
	\Delta_{\PP(A)_{0}}: A^{2} &\longrightarrow A^{2}\underset{A}\otimes A^{2} \notag \\
	a\otimes b &\longmapsto a\otimes 1\underset{A}\otimes 1\otimes b \\
	\Delta_{\PP(A)_{1}}: A^{2} &\longrightarrow A^{2}\otimes_{A} A^{2} \oplus A^{2}\underset{A}\otimes A^{2} \notag \\
	a\otimes b &\longmapsto (a\otimes 1\underset{A}\otimes 1\otimes b, a\otimes 1\underset{A}\otimes 1\otimes b) \notag
\end{align}
By direct computations, it follows that $\Delta_{\PP(A)}$ is a diagonal approximation map.

\begin{pro}\label{isoalgZ}
As algebras,
$$
	HH^{*}(R[\Z];R[\Z])\cong R[x,x^{-1}]\otimes \Lambda(y)
$$
where $x,x^{-1}\in HH^{0}(A;A)$ and $y\in HH^{1}(A;A)$.
\end{pro}

\begin{proof}
Using the diagonal approximation map (\ref{diagA}), it can be checked that the cup product is given by multiplication in degrees $0$ and $1$, and $0$ in degrees greater than $2$. Therefore, taking $x,x^{-1}\in HH^{0}(A;A)$ to be $t,t^{-1}\in A$ and $y\in HH^{1}(A;A)$ to be $1\in A$, we get the desire isomorphism of algebras.
\end{proof}

From the definition of the action \ref{actionHH} and the diagonal map \ref{diagA} follows that

\begin{lem}
The action of $HH^*(A;A)$ on $HH_1(A;A)$ is given by 
\begin{align}
\rho: HH_{1}(A;A)\otimes HH^*(A;A)&\longrightarrow HH_{1-*}(A;A) \notag \\
		a\otimes b&\longmapsto (-1)^{|b|} ab \notag 
\end{align}
\end{lem}

Let $\psi: \PP(A) \rightarrow \BB(A)$ and $\varphi: \BB(A) \rightarrow \PP(A)$ be the chain maps defined as follows
\[
\xymatrix{ \cdots \ar[r] & 0 \ar[r]^{0} \ar@<-0.5ex>[d]_{0} & A^{2} \ar[r]^{d_1} \ar@<-0.5ex>[d]_{\psi_1} & A^{2} \ar[r]^(0.56){\mu} \ar@<-0.5ex>[d]_{\psi_0} & A \ar[r] \ar@{=}[d] & 0\\
 \cdots \ar[r] & A\otimes \bar{A}^{2}\otimes A \ar[r]^(0.53){\partial_{2}} \ar@<-0.5ex>[u]_{0} & A\otimes \bar{A}\otimes A \ar[r]^(0.64){\partial_{1}} \ar@<-0.5ex>[u]_{\varphi_1} &  A^{2} \ar[r]^(0.56){\partial_{0}} \ar@<-0.5ex>[u]_{\varphi_0} & A \ar[r]& 0 }
\]

\begin{align}
	\psi_0 &\equiv id & \varphi_0 &\equiv id \notag \\
	\psi_{1}(1\otimes 1) &= -1\otimes t\otimes 1 & \varphi_{1}(1\otimes t^{k} \otimes 1)&=\begin{cases}
									-\displaystyle\sum_{j=0} ^{k-1} t^j \otimes t^{k-j-1} &\text{if $k\geq 1$} \notag\\
									 0 &\text{if $k=0$} \notag \\
									\displaystyle\sum_{j=0} ^{-k-1} t^{-j-1} \otimes t^{k+j} &\text{if $k\leq -1$} \notag
								\end{cases} \notag \notag
\end{align} 


\begin{pro}\label{iden}
Using the identifications $A^{n+2} \underset{A^e}\otimes A\cong A^{n}\otimes A$ and $A^{2}\underset{A^e}\otimes A\cong A$ \\
the induced maps for $\psi_*$ and $\varphi_*$ are
\begin{align}
	\bar{\psi}_0 &\equiv id & \bar{\varphi}_0 &\equiv id \notag \\
	\bar{\psi}_1:A &\rightarrow A\otimes A & \bar{\varphi}_1:A\otimes A &\rightarrow A \notag \\
	a &\mapsto -a\otimes t & a\otimes t^k &\mapsto -kat^{k-1} \notag \
\end{align} 
Using the identifications $Hom_{A^{e}} (A^{n+2}, A)\cong  Hom (A^{n}, A)$ and $Hom_{A^{e}}(A^{2} , A)\cong A$ \\
the induced maps for $\psi_*$ and $\varphi_*$ are
\begin{align}
	\bar{\psi}^*_0 &\equiv id & \bar{\varphi}^*_0 &\equiv id \notag \\
	\bar{\psi}^*_1:Hom(A,A) &\rightarrow A & \bar{\varphi}^*_1:A &\rightarrow Hom(A,A) \notag \\
	f &\mapsto -f(t) & a &\mapsto f_a: A\rightarrow A \notag \\
	& & & \qquad \quad \; t^k\mapsto -kat^{k-1} \notag 
\end{align}
\end{pro}

The BV-structure on Hochschild cohomology of the group ring of the integers is given by

\begin{teo}\label{BVKZ}
Let $a=ut^k$ with $u\in R^{\times}$ and $k\in\Z$. As a BV-algebra,
\begin{align}
	HH^{*}(R[\Z];R[\Z])&\cong R[x,x^{-1}]\otimes \Lambda(y) \notag \\
\Delta_a (x^{i}) &= 0  \notag \\ 
	\Delta_a (yx^{i}) &= (i+k)x^{i-1}  \notag
\end{align}
where $x,x^{-1}\in HH^{0}(A;A)$ and $y\in HH^{1}(A;A)$.
\end{teo}  

\begin{proof}
Let $a\in HH_1(A;A)\cong R[\Z]$ and $\rho_a$ be the map defined as follows 
\begin{align}
\rho_a:HH^*(A;A)&\longrightarrow HH_{1-*}(A;A) \notag \\
		b&\longmapsto \rho(a\otimes b) \notag 
\end{align}
Since the action is given by multiplication, $\rho_a$ is an isomorphism for any unit $a\in R[\Z]$. Even more, any unit in $R[\Z]$ is of the form $a=ut^k$ with $u\in R^{\times}$ and $k\in\Z$. By theorem \ref{HHBVact}, $HH^*(A;A)$ is a BV-algebra and the BV-operator $\Delta_a$ is given by
\[
\xymatrix{
\Delta_a: HH^*(A;A) \ar[d]_{\rho_a} \ar[r] & HH^{*-1}(A;A) \\
HH_{1-*}(A;A) \ar[r]^-{B} & HH_{1-(*-1)}(A;A) \ar[u]_{\rho^{-1}_a}
}
\] 
By degree reasons $\Delta_a(x^i)=0$ and $\Delta_a(yx^i)$ is given by
$$
 t^i \xmapsto{\rho_a} -ut^{i+k} \xmapsto{\bar{\psi}^*_{0}} -ut^{i+k} \xmapsto{B} -u\otimes t^{i+k} \xmapsto{\bar{\varphi}_{1}} u(i+k)t^{i+k-1} \xmapsto{\rho^{-1}_a} (i+k)t^{i-1} 
$$ 

\end{proof}

In \cite{M3}, Menichi calculates the BV-algebra structure of the homology of the free loop space of $\Ss$

\begin{teo}[\cite{M3} Theorem 10]
As a BV-algebra,
\begin{align}
	\HH_{*}(L\Ss;R)&\cong R[x,x^{-1}]\otimes \Lambda(z) \notag \\
\Delta(x^{i}) &= 0  \notag \\ 
	\Delta(zx^{i}) &= ix^{i}  \notag
\end{align}
where $|x|=0$ and $|z|=-1$.
\end{teo}

This BV-algebra and the BV-algebra of the Hochschild cohomology of the group ring of the integers are related by 

\begin{cor}
There is an isomorphism of BV-algebras
$$
	\phi: \HH_{*}(L\Ss;R)\xrightarrow{\cong} HH^{*}(R[\Z];R[\Z])
$$
\end{cor}

\begin{proof}
By theorem \ref{BVKZ}, $HH^{*}(R[\Z];R[\Z])$ can be endowed with many BV-algebra structures as units in $R[\Z]$. For the existence of this isomorphism, we are considering the BV-operator given by the unit $a=t^{-1}$. Then as a BV-algebra 
\begin{align}
	HH^{*}(R[\Z];R[\Z])&\cong R[x,x^{-1}]\otimes \Lambda(y) \notag \\
	\tilde{\Delta} (y^r x^{i}) &= r(i-1)x^{i-1}  \notag
\end{align} 
The isomorphism $\phi$ is defined as follows 
$$
	\phi(x)=x \qquad \text{and} \qquad \phi(z)=yx 
$$
It is clear that $\phi$ is an isomorphism of graded algebras, and
$$
\phi\Delta(z^r x^i) = \phi(rix^i) = rix^{i} =r(i+r-1)x^{i+r-1}= \tilde{\Delta}(y^rx^{i+r}) = \tilde{\Delta}\phi(z^r x^i) 
$$
\end{proof}


Since $HH^{*}(R[\Z];R[\Z])$ is $R$-projective and the resolution $\PP(A)$ (\ref{Pres}) satisfies the conditions of theorem \ref{injection}. By theorem \ref{BVisoCY}, we get 

\begin{teo}\label{BVfiniterank}
As BV-algebras,
\begin{align}
	HH^{*}(R[\Z^{n}]&;R[\Z^{n}])= R[x_1,x_1 ^{-1},\dots,x_n,x_n ^{-1}]\otimes \Lambda(y_1,\dots ,y_n) \notag \\
	\Delta(x_1 ^{i_1}\cdots x_n ^{i_n}y_1 ^{r_1}\cdots y_n ^{r_n}) &= \displaystyle \sum_{k=1} ^{n} (-1)^{^{r_1+\cdots +r_{k-1}}} r_k( i_k-1) x_1 ^{i_1}\cdots x_i^{i_k-1}\cdots x_n ^{i_n}y_1 ^{r_1}\cdots \widehat{y_k ^{r_k}}\cdots y_n ^{r_n}\notag
\end{align}
where $|x_i|=|x_i ^{-1}|=0$ and $|y_i|=1$ for $1\leq i\leq n$.
\end{teo}

As a corollary, we have 

\begin{cor}
As  Gerstenhaber algebras,
$$
	HH^{*}(R[\Z^{n}];R[\Z^{n}])= R[x_1,x_1 ^{-1},\dots,x_n,x_n ^{-1}]\otimes \Lambda(y_1,\dots ,y_n)
$$
where $|x_i|=|x_i ^{-1}|=0$ and $|y_i|=1$ for $1\leq i\leq n$. The bracket is generated by
$$
	\lbrace x^{r}_i,x^{s}_j\rbrace =0, \qquad \lbrace y_i,y_j\rbrace=0, \qquad \text{and} \qquad \lbrace x^{r}_i,y_j\rbrace=-r\delta_{ij}x_i^{r-1}
$$ 
\end{cor}

Let $G$ be a finitely generated abelian group. Then $G$ can be decomposed as $G \cong \Z^{n}\oplus H
$ with $H$ a finite abelian group. Therefore,
$$
	R[G] \cong R[\Z^n]\otimes R[H]
$$
By theorem \ref{AxB}, there is an isomorphism of Gerstenhaber algebras 
$$
	HH^*(R[G];R[G]) \cong HH^*(R[\Z^{n}];R[\Z^{n}]) \otimes HH^*(R[H]; R[H])
$$

\begin{cor}
Let $G$ be a finitely generated abelian group. Then as a BV-algebra 
$$
	HH^*(R[G];R[G]) \cong HH^*(R[\Z^{n}];R[\Z^{n}]) \otimes HH^*(R[H]; R[H])
$$
with BV-operator given by 
$$
	\Delta= \Delta^{\Z^n}\otimes id \pm id \otimes \Delta^{H}
$$
where $\Delta^{\Z^n}$ is given by theorem \ref{BVfiniterank} and $\Delta^H$ is the BV-operator for the finite group $H$.
\end{cor}

\bibliography{refs}      
\bibliographystyle{IEEEtranS}  

\end{document}